\documentclass[12pt,a4paper,reqno]{amsart}
\usepackage[
left=1in,right=1in,top=1in,bottom=1in,
]{geometry}

\usepackage{hyperref}
\hypersetup{
    colorlinks=true,
    linkcolor=blue,
    citecolor=blue,
    urlcolor=blue,
    pdfauthor={},
    pdftitle={Collapsing singular Riemannian foliations with flat leaves}
}

\usepackage{mathscinet}
\usepackage[numbers]{natbib}
\usepackage{url}
\usepackage[dvipsnames]{xcolor}
\usepackage{graphicx}
\usepackage{amsfonts, amssymb, amsmath, amsthm}
\usepackage{mathtools}
\usepackage{multicol}
\usepackage{tikz}
\usepackage{tikz-cd}
\usepackage{tikz-3dplot}
\usepackage{cancel}
\usepackage{changepage}
\usepackage{csquotes}
\usepackage{adjustbox}
\usetikzlibrary{matrix}
\usepackage{enumerate}
\usepackage[normalem]{ulem}
\usepackage{verbatim}
\usepackage{mathrsfs}

\newtheorem{mtheorem}{Theorem}  
 
\newtheorem{mcorollary}[mtheorem]{Corollary}

\newtheorem{theorem}{Theorem}[section]
\newtheorem{lemma}[theorem]{Lemma}
\newtheorem{prop}[theorem]{Proposition}

\newtheorem{question}{Question}

\theoremstyle{definition}

\newtheorem{ex}[theorem]{Example}
\newtheorem{rmk}[theorem]{Remark}

\numberwithin{equation}{section}

\usepackage[reftex]{theoremref}

\newcommand{\Alex}{\mathrm{Alex}}

\newcommand{\fol}{\mathcal{F}}
\newcommand{\F}{\mathsf{F}}

\newcommand{\Gr}{\mathrm{Gr}}
\newcommand{\Hol}{\mathrm{Hol}}

\newcommand{\N}{\mathbb{N}} 

\newcommand{\Ns}{\mathsf{N}}
\newcommand{\Or}{\mathrm{O}}

\newcommand{\prin}{\mathrm{prin}}
\newcommand{\R}{\mathbb{R}} 

\newcommand{\reg}{\mathrm{reg}}

\newcommand{\Sp}{\mathbf{S}} 

\newcommand{\Sec}{\mathrm{Sec}}

\newcommand{\X}{\mathfrak{X}}

\DeclareMathOperator{\curv}{curv}

\DeclareMathOperator{\diam}{Diam}
\DeclareMathOperator{\Id}{Id}
\DeclareMathOperator{\Iso}{Isom}

\DeclareMathOperator{\Inj}{Injrad}

\DeclareMathOperator{\dime}{dim}

\DeclareMathOperator{\expo}{exp}

\usepackage{lineno}

\usepackage[colorinlistoftodos, textwidth=3cm]{todonotes}
\newcounter{dccomment}

\begin{document}
\title[Collapsing flat foliations]{Collapsing regular Riemannian foliations with flat leaves}

\author[D. CORRO]{Diego Corro$^{*}$}
\address[D. CORRO]{
Fakultät für Mathematik\\
Karslruher Institut für Technologie\\
Englerstr. 2\\
76131 Karlsruhe\\
Deuschland.
}
\curraddr{School of Mathematics, Cardiff University, Cardiff, UK.}
\email{diego.corro.math@gmail.com}
\thanks{$^{*}$This work was supported by a UKRI Future Leaders Fellowship [grant number MR/W01176X/1; PI: J Harvey], by the DFG (281869850, RTG 2229 ``Asymptotic Invariants and Limits of Groups and Spaces''), a DGAPA postdoctoral Scholarship of the Institute of Mathematics of UNAM, and DFG-Eigene\-stelle Fellowship CO 2359/1-1 as part of the DFG-SPP 2026 ``Geometry at Infinity''.}


\setlength{\overfullrule}{5pt}

\begin{abstract}
In this manuscript we present how to collapse a manifold equipped with a closed flat regular Riemannian foliation with leaves of positive dimension, while keeping the sectional curvature uniformly bounded from above and below. From this deformation, we show that  in the case when the manifold is compact and simply connected the foliation is given by torus actions. This gives a geometric characterization of aspherical regular Riemannian foliations given by torus actions.
\end{abstract}

\maketitle


%

\section{Introduction}\label{S: Introduction}

Effective actions by compact Lie groups are a classical example of symmetry in differential topology and differential geometry. In particular a lot of consideration has been given to torus actions since they correspond to the abelian part of a Lie group. Recently singular Riemannian foliations have been considered as a generalized form of symmetry. These are smooth foliations with  a Riemannian metric that is ``bundle like''. The justification for considering singular Riemannian foliations as a notion of symmetry comes from the fact that Lie group actions by isometries and Riemannian submersions are examples of such foliations.

In the setting of singular Riemannian foliations, Galaz-García and Radeschi in \cite{GalazGarciaRadeschi2015} introduced an analogous concept to torus actions via isometries by considering singular Riemannian foliations whose leaves are closed and aspherical manifolds. These foliations are called \emph{$A$-foliations} and examples of such foliations include effective torus actions by isometries. In general due to \cite{GalazGarciaRadeschi2015} and \cite{Corro2023}, on a simply-connected manifold the leaves of an $A$-foliation are finitely covered by tori, and  moreover, the leaves of highest dimension are all homeomorphic to a torus.  Thus in this way, $A$-foliations resemble torus actions.  But in \cite{FarrellWu2019} examples are given of $A$-foliations, on closed manifolds with a non-trivial finite fundamental group, which are not given by torus actions. From this discussion,  it follows that it is natural to ask if we may characterize those $A$-foliations given by a torus action. 

In \cite{GalazGarciaRadeschi2015} the authors also introduced the notion of a \emph{$B$-foliation}. These are $A$-foliations such that the leaves are homeomorphic to manifolds that admit a flat metric, i.e. they are homeomorphic to \emph{Bieberbach manifolds}. In \cite[Theorem C]{Corro2023} it was proven that for  a simply-connected manifold the leaves of any $A$-foliation are homeomorphic to Bieberbach manifolds, except possibly for the singular leaves of dimension $4$. Due to this observation,  it was proposed in \cite{Corro2023} to redefine \emph{$B$-foliations} as those $A$-foliations $(M,\fol,g)$ such that for any leaf $L\in \fol$ we have that  the leaf $(L,g|_{L})$ with the induced metric is a flat manifold (in particular the leaves are Bieberbach manifolds). 

In this work we show that  on a simply-connected manifold under the updated definition, such $B$-foliations with leaves of positive constant dimension are given by torus actions.

\begin{mtheorem}\th\label{MT: flat foliations are given by group actions}
Consider $(M,\fol,g)$ a regular $B$-foliation on a compact simply-connected Riemannian manifold. Then the foliation $\fol$ is induced by an isometric torus action.
\end{mtheorem}

To prove \th\ref{MT: flat foliations are given by group actions} we compare regular $B$-foliations to the  so called $\F$-structures and $\Ns$-structures introduced by Cheeger, Gromov and Fukaya in \cite{CheegerGromov1986} and \cite{CheegerFukayaGromov1992}.  These are other types of foliations, and were introduced  when studying the notion of  ``collapse with bounded curvature'' onto metric spaces. We present now the definition of $\F$-structures (for the definition of $\Ns$-structures see Section~\ref{SS: N- and F-structures}). We recall that an $\F$-structure is given by an open cover $\{U_i\}_{i\in \mathfrak{I}}$ of $M$, together with a $T^{k_i}$-action on a finite  normal (or Galois) cover $\pi_i\colon \tilde{U}_i\to U_i$ for each index, such that the torus actions satisfy certain compatibility conditions when passing from one neighborhood to another, and they are compatible with the action of the deck-transformation group of $\pi_i$ (see for example \cite[Definition 19.2]{Fukaya1990}). An $\F$-structure induces a partition of $M$ into submanifolds called the orbits of the $\F$-structure. We say that a sequence of Riemannian metrics $\{g_n\}_{n\in \N}$ on a compact manifold $M$ \emph{collapses with bounded curvature} if there exist constants $\lambda\leq \Lambda\in \R$ such that
\[
\lambda\leq \Sec(g_n)\leq \Lambda,
\]
and for any $p\in M$ the injectivity radius of $g_n$ at $p$ goes to $0$ as $n\to \infty$ (see for example \cite[p. 69]{Pansu1985}). By \cite{CheegerGromov1986} and \cite{CheegerGromov1990} we have that $\F$-structures  characterize collapse with bounded curvature: A compact manifold collapses with bounded curvature if and only if it admits an $\F$-structure whose orbits have positive dimension. 

For the proof of \th\ref{MT: flat foliations are given by group actions}, we prove that given a regular $B$-foliation $(M,\fol,g)$ on a compact manifold with leaves of positive dimension we may always collapse $M$ with bounded curvature and bounded diameter.

\begin{mtheorem}\th\label{MT: Collapse of flat foliations}
Consider $(M,\fol,g)$ a regular non-trivial $A$-foliation on a compact Riemannian manifold, such that each leaf $L\in \fol$ with the induced Riemannian metric $g|_L$ is flat, i.e. a $B$-foliation. Then by shrinking the directions tangent to the leaves of $\fol$  we may collapse with bounded curvature and uniformly bounded diameter the manifold $(M,g)$.In particular we show there exists a family of Riemannian metrics $\{g_\delta\}_{1\geq\delta>0}$ on $M$ such that $\lambda\leq \Sec(g_\delta)\leq \Lambda$, $\diam(g_\delta)\leq D$ and $\Inj(g_\delta)\to 0$ as $\delta\to 0$. 
\end{mtheorem}

We recall that in the study of collapse  with bounded  curvature in \cite{CheegerFukayaGromov1992}, given such a collapsing sequence the authors construct $\Ns$-structures and highly symmetric $\Ns$-invariant Riemannian metrics which are bounded by the collapsing sequence (see \th\ref{T: Bounded collapse in 1-connnected manifolds}). When the manifold $M$ has finite fundamental group an $\Ns$-structure is given by an $\F$-structure (see \th\ref{T: N-structure over manifold with finite fundamental group are given by T-structure}). 
In the case when $M$ has finite fundamental group we use the relationship between $\Ns$-structures and $\F$-structures, the invariant metrics and the explicit collapsing sequence given in the proof of \th\ref{MT: Collapse of flat foliations} to show that the geometry of  $(M,\fol,g)$ where $\fol$ is a regular $B$-foliation with leaves of positive dimension is approximated by geometries invariant under $\F$-structures. 
We can also point to a relationship between the tangent spaces of the leaves of the foliations and the tangent spaces of the orbits of the $\F$-structure.

\begin{mtheorem}\th\label{MT: Approximating flat foliations by N-structures}
Consider $(M,\fol,g)$ a regular non trivial $A$-foliation on a compact $m$-dimensional Riemannian manifold with finite fundamental group and $m>2$, such that each leaf $L\in \fol$ with the induced Riemannian metric $g|_L$ is flat, i.e. a $B$-foliation. Then there exists a sequence of Riemannian metrics $\{\bar{g}^\varepsilon_N(\varepsilon)\}_{1\geq\varepsilon>0}$ with each metric invariant under an $\F$-structure, and such that as $\varepsilon\to 0$, we have that $(M,\bar{g}^\varepsilon_N(\varepsilon))$ converges in the Gromov-Hausdorff sense to $(M,g)$ as $\varepsilon\to 0$. Moreover, given $p\in M$ fixed and taking $x\in T_e T^k$ fixed in the Lie algebra of the torus tracing the leaves of the $\F$-structure around $p$, we have that the action field $X^\ast(p)$ converges to a direction tangent to the leaf $L_p$.
\end{mtheorem}

\begin{rmk}
By \th\ref{L: uniform convergence of bar g to g } we also have that the Riemannian tensors $\bar{g}_{N(\varepsilon)}^\varepsilon$ converge uniformly to $g$.
\end{rmk}

We remark that by \th\ref{MT: Collapse of flat foliations}, given a regular $B$-foliation on a compact manifold $M$ with leaves of positive dimension  we have a continuous family $g_\delta$ of Riemannian metrics collapsing with bounded curvature to the leaf space. Combining this with  \cite[Theorem 0.4]{PetruninRongTuschmann1999} we conclude that we  have an upper bound for the minimum of the sectional curvature of a compact manifold equipped with a $B$-foliation whose leaves have positive dimension. 

\begin{mcorollary}
Let $(M,\fol,g)$ be a regular $B$-foliation on a compact manifold with leaves of positive dimension. Then we have for $\lambda$ equal to the minimum of the sectional curvature of $g$, that $\lambda\leq 0$.
\end{mcorollary}

We recall that an $\F$-structure is \emph{pure}, if for any point $p\in M$ in the intersection  of two neighborhoods $U_i\cap U_j$ given in  definition of the $\F$-structure, the images containing $p$ of the orbits of the torus actions under $\pi_i$ and $\pi_j$ coincide (see \cite[p. 228]{Fukaya1990}). By \cite{CheegerGromov1986}, pure $\F$-structures are examples of $B$-foliations as defined in this work. This fact motivated the following question in \cite{GalazGarciaRadeschi2015}:

\begin{question}
How do the concepts of $A$-foliations and $\F$-structures relate to each other?
\end{question}

\th\ref{MT: flat foliations are given by group actions}  sheds some light on this question: Regular flat $A$-foliations (i.e. $B$-foliations) on simply-con\-nec\-ted manifolds are given by torus actions, and thus are $\F$-structures.

We highlight that on a simply-connected manifold, an $\F$-structure is induced by a torus action (see for example \cite[Proposition 3.1, Lemma 3.2 and Lemma 4.1]{Rong1996}). Thus it is a pure $\F$-structure, and therefore a $B$-foliation. Still such $B$-foliations may have singular leaves. For example torus actions of cohomogeneity $2$ can have singular leaves \cite{Oh1983}, \cite{Corro2023}. Nonetheless, when $M$ has finite fundamental group  a regular $A$-foliation with flat leaves is the limit of $\F$-structures by \th\ref{MT: Approximating flat foliations by N-structures}.

We point out that not all regular $A$-foliations admit a Riemannian metric making it a $B$-foliation. In \cite{FarrellWu2019} the authors construct manifolds of dimension $m\geq 9$ with non-trivial finite fundamental group and each equipped with a regular $A$-foliation whose leaves are homeomorphic to the $(m-4)$-dimensional torus but not diffeomorphic to the $(m-4)$-dimensional torus. These foliations do not admit any Riemannian metric making them into a $B$-foliation, since by  \cite{Bieberbach} a flat  manifold is homeomorphic to the torus if and only if it is diffeomorphic to the torus. 

It would be interesting to know if there exists an example of an $A$-foliation on a compact simply-connected manifold whose leaves are exotic tori. Since torus actions induce $B$-foliations, a more general question is if an $A$-foliation on a simply-connected manifold can be made into a $B$-foliation by a change of the metric.

\begin{question}\th\label{Q: Question 2}
Is there an $A$-foliation on a simply-connected manifold which is not given by a torus action? 
\end{question}

From our results, this is related to the following weaker question for foliations with leaves of positive dimension: Given any  $A$-foliation on a simply-connected manifold,  is there an other foliated Riemannian metric $\bar{g}$ on $M$ making it into a $B$-foliation?

In \cite[Theorem A]{Corro2023} a negative answer to \th\ref{Q: Question 2} is given for $A$-foliations of codimension $2$. Namely in \cite{Corro2023}, the author explored the problem of whether we can compare the diffeomorphism type of two compact simply-connected manifolds with $A$-foliations via their leaf spaces equipped with local information of the foliation, focusing on the case of $A$-foliations of codimension $2$. The positive answer given in \cite[Theorem C and Lemma 5.11]{Corro2023} to this comparison problem together with the classification of tours actions of cohomogeneity $2$ in \cite{Oh1983}, allows us to conclude that an $A$-foliation of codimension $2$ on a compact simply-connected manifold is given up to a foliated homeomorphism by a torus action. Then, by analyzing the local information of the foliation we are able to improve the conclusion to a foliated diffeomorphism. 

A crucial point in the proof of \cite[Theorem A]{Corro2023} is that for an $A$-foliation of codimension $2$ on a compact simply-connected manifold the stratification of the leaf space  is simple and can be explicitly described. In general, even for singular Riemannian foliations on  vector spaces with finite dimension, it is not clear how to compare the foliations via their leaf spaces (see  \cite[Question 1.2]{GorodskiLytchak2014}). We point out that the proof presented in \cite{Corro2023} is of a topological nature and only uses the fact that the presence of a foliated Riemannian metric guarantees a nice local structure of the foliation on  tubular neighborhoods of leaves (see \cite{MendesRadeschi2019}). Moreover, to be able to extend the ideas in \cite{Corro2023} to higher dimensions, we would require to impose extra topological conditions on $M$ and on the leaf space of the foliation. Thus, to the best of our knowledge, the approach presented in this manuscript is a different approach for comparing foliated structures to what exists in the literature.

The present work was motivated by the need to give geometric conditions on $A$-foliations that force them to be given by torus actions. Due to \cite{CheegerFukayaGromov1992} and \cite{Rong1996} it follows that on simply-connected manifolds, torus actions with orbits of positive dimension are characterized by the phenomenon of collapse with bounded curvature (see \th\ref{T: Bounded collapse in 1-connnected manifolds}). Thus, a natural approach to finding sufficient geometric conditions for an $A$-foliation to be induced by a group action, is to  find sufficient geometric conditions which let us collapse with bounded curvature the manifold by shrinking  the leaves of a given $A$-foliation on a simply-connected manifold. 

We highlight  that to the best of our knowledge in the literature there are few references on the subject of deforming foliations. In \cite{FarrellJones1998} the authors show that for a particular class of $1$-dimensional regular foliations, a portion of the manifold has a cover into ``long and thin'' submanifolds extending the notions in \cite{CheegerFukayaGromov1992} to the foliated setting. In \cite{CranicMestreStruchiner2020} the authors consider a parametrized family of Lie groupoids and give  rigidity results for such ``deformations'' of Lie groupoids. Lie groupoids are closely related to regular foliations (see \cite{MoerdijkMrcun2003}), and also singular Riemannian foliations (see \cite{AlexandrinoInagakideMeloStruchiner2022}). Applying these rigidity results for Lie groupoids,  the authors in \cite{delHoyoFernandes2019} consider on a compact manifold a family of regular foliations  $\fol_t$ parametrized by a real parameter $t\in [0,1]$, each having compact leaves and Hausdorff leaf space, and show that at $t=0$ the existence of a leaf $L$ with $H^1(L,\R) =0$ is a sufficient condition for such a parametrization to be rigid, i.e. $\fol_t = \fol_0$ for all $t\in [0,1]$. This is an analogous result to the conclusions presented in \cite{EpsteinRosenberg1977} and in \cite{Hamilton1978}:  Consider two regular foliations $\fol_1$, and $\fol_2$ with compact leaves on a fixed  manifold $M$ without boundary, such that they are close to each other in an appropriate sense, and such that the leaf spaces of the foliations are Hausdorff  spaces. In the case that there exists a leaf $L$ in one of the  foliations with $H^1(L,\R) =0$,  there exists a foliated diffeomorphism between the foliated manifolds. For group actions there is a general rigidity result by Grove and Karcher \cite{GroveKarcher1973}: Given a compact Lie group $G$, a compact manifold $M$, and two group actions $\mu_1\colon G\to\mathrm{Diff}(M)$ and $\mu_2\colon G\to \mathrm{Diff}(M)$, if the images $\mu_1(G)$ and $\mu_2(G)$ are $C^1$ close enough in $\mathrm{Diff}(M)$, then the images are conjugated in $\mathrm{Diff}(M)$.

We point out that the proof of \th\ref{MT: flat foliations are given by group actions} is also a  rigidity result for a deformation procedure, since the proof consist of approximating our foliation by torus actions using \th\ref{MT: Approximating flat foliations by N-structures}, and showing that there exists a limit torus action whose orbits agree with the leaves of the foliation. Nonetheless, we point out that  to the best of our knowledge there are no general rigidity results for singular Riemannian foliations as those presented in \cite{delHoyoFernandes2019}, \cite{EpsteinRosenberg1977}, and \cite{Hamilton1978}. Moreover, the rigidity result presented in \th\ref{MT: flat foliations are given by group actions} agree with the conclusions in \cite{GroveKarcher1973}. One can also compare \th\ref{MT: flat foliations are given by group actions} with \cite[Proposition 0.8]{PetruninRongTuschmann1999} and comments after. In our work we also point to the existence of a fixed torus action on $M$, but we do show that we can take our foliated metric as the limit of the perturbed collapsing sequence. 

Our article is organized as follows: In Section~\ref{S: Preliminaries} we review the necessary preliminaries on singular Riemannian foliations, as well as the concepts of $A$-foliations, $B$-foliations, and $\Ns$-structures. We also review Alexandrov spaces and equivariant convergence of metric spaces. At the end of Section~\ref{S: Preliminaries}, we present the relevant theorems of collapse theory used in our proofs. In Section~\ref{S: Proofs of main theorems} we give the proofs of our main theorems. We begin by presenting the proof of \th\ref{MT: Collapse of flat foliations}, and then proceed with the proof of \th\ref{MT: Approximating flat foliations by N-structures}. We use then these results and their proofs to  prove  \th\ref{MT: flat foliations are given by group actions}.

\section*{Acknowledgments}

I thank Wilderich Tuschmann and Karsten Grove for discussions on $\F$-structures, and I thank Fernando Galaz-Garcída and Alexander Lytchak for discussions on collapse theory. I also thank John Harvey for comments that improved the proofs of the main theorems. Last I thank Jesús Núñez-Zimbrón and Jaime Santos for pointing me to the results about equivariant convergence. I also thank the anonymous referee for suggestions that improved the presentation of the proofs.

\section{Preliminaries}\label{S: Preliminaries}

In this section we present the preliminaries necessary for our results. We begin by presenting the definition of a singular Riemannian foliation, then present the notion of the infinitesimal foliation and the holonomy of a leaf. We then present the notions of $A$- and $B$-foliations. We also present the necessary results from the theories of convergence of Riemannian manifolds  with curvature bounds, equivariant convergence of metric spaces, and metric spaces with a lower curvature bound.

\subsection{Singular Riemannian foliations}\label{SS: Singular Riemannian foliations}

A \emph{singular Riemannian foliation} $(M,\fol,g)$ on a Riemannian manifold $(M,g)$ is a partition $\fol$ of $M$ into connected injectively immersed submanifolds $L\in\fol$ called \emph{leaves}, such that the following hold:
\begin{enumerate}[(i)]
\item There exists a family of smooth vector fields $\{X_\alpha\}_{\alpha\in \mathcal{I}}\subset \mathfrak{X}(M)$ on $M$, such that for each point $p\in M$ the vector fields span the tangent space at $p$ of the leaf $L_p$ containing $p$.
\item Given $\gamma\colon [0,1]\to M$ a geodesic perpendicular to $L_{\gamma(0)}$, then $\gamma$ is perpendicular to $L_{\gamma(t)}$, for all $t\in [0,1]$.
\end{enumerate}

The first condition insures that the singular distribution $\Delta$ in $TM$ given by $\Delta(p) = T_p L_p$ is smooth. The second condition is equivalent to the leaves of  $\fol$ being locally equidistant. When the leaves of the foliation are closed submanifolds, we say that \emph{the foliation is closed}. We point out that in this case the leaves are embedded submanifolds.

We define the leaves of maximal dimension to be \emph{regular leaves}, and the ones that do not have  maximal dimension to be \emph{singular leaves}. The \emph{codimension of the foliation} $\mathrm{codim}(\fol)$ is equal to the codimension of any regular leaf in $M$, and the \emph{dimension of $\fol$}, denoted as $\dim(\fol)$, to be the dimension of the regular leaves.

For a closed singular Riemannian foliation $(M,g,\fol)$ and $0\leq \ell\leq \dim(\fol)$ we define the $\ell$-dimensional stratum as
\[
\Sigma_\ell = \{p\in M\mid \dim(L_p) = \ell\},
\] 
and point out that each connected component $C$ of $\Sigma_\ell$ is an embedded submanifold of $M$ (see \cite{Radeschi2012}). We denote  by $\Sigma^p$ the connected component of $\Sigma_\ell$ containing $p$, where $\ell= \dim(L_p)$. The collection $\{\Sigma^p\mid p\in M\}$ give the \emph{canonical stratification of $M$ by $\fol$}. For the top dimension $\dim(\fol)$, the set $\Sigma_{\dim(\fol)}$ is open, dense and connected. We refer to it as the \emph{regular stratum}, and denote it by $M_\reg = \Sigma_{\dim(\fol)}$. In the case when $M_\reg = M$ we say that the foliation $\fol$ is \emph{regular}. We also remark that all other strata $\Sigma_\ell$ have codimension at least $2$ (\cite[Section 2.3]{GalazGarciaRadeschi2015}).

We equip the \emph{quotient space} $M/\fol$ with the quotient topology, making the \emph{quotient map} $\pi\colon  M\to M/\fol$ continuous. We refer to $M/\fol$ equipped with the quotient topology as the \emph{leaf space of $\fol$}. When the foliation is closed and $M$ is complete, the Riemannian distance on $M$ induces an inner metric, called the \emph{quotient metric} on the leaf space, and it induces the quotient topology. Moreover as pointed out in \cite[p. 2943]{GroveMorenoPetersen2019}, for a closed foliation on a complete manifold, the leaf space equipped with the quotient metric is locally an Alexandrov space (see Section~\ref{SS: Alexandrov geometry} for a definition of an Alexandrov space).

Given a singular Riemannian foliation $\fol$ on a vector space $V$ equipped with an inner product $g$, we say that $(V,\fol,g)$ is an \emph{infinitesimal foliation} if for the origin $\bar{0}\in V$ we have $L_{\bar{0}}=\{\bar{0}\}$. 

We now consider $(M,\fol,g)$ a singular Riemannian foliation on a complete Riemannian manifold, and fix $p\in M$. We denote by $\nu_p(M,L_p)=  \{v\in T_p M\mid g(p)(v,x) =0\mbox{ for any } x\in T_p L_p\}$ the set of \emph{normal space to the leaf at $p$}. Given $\varepsilon>0$ we consider the \emph{closed normal disk of radius $\varepsilon$} defined as $\nu_p^\varepsilon(M,L_p) =\{v\in \nu_p(M,L_p)\mid \|v\|^2_g \leq \varepsilon\}$. Then by fixing $\varepsilon>0$ small enough (for example $\varepsilon$ smaller than the injectivity radius at $p$), we may assume that $\expo_p\colon \nu_p^\varepsilon(M,L_p)\to S_p = \exp_p(\nu_p^\varepsilon(M,L_p))\subset M$ is a diffeomorphism. For $v\in \nu_p^\varepsilon(M,L_p)$ we define $\mathcal{L}_v\subset \nu_p^\varepsilon(M,L_p)$ to be the connected component of \linebreak$\expo_p^{-1}(L_{\expo_p(v)}\cap S_p)$ containing $v$. In this way, we obtain a partition $\fol_p(\varepsilon)$ on $\nu_p^\varepsilon(M,L_p)$. By \cite[Lemma 6.2]{Molino} the foliation $\fol_{p}(\varepsilon) $ is invariant under homotheties $h_{\lambda}\colon \nu_p^\varepsilon(M,L_p)\to \nu_p^{\lambda\varepsilon}(M,L_p)$, with $h_{\lambda}(v) = \lambda v$. From this we see that $\fol(\varepsilon)$ is independent of $\varepsilon$, and thus we can extend  $\fol_p(\varepsilon)$ to a partition $\fol_p$ on $\nu_p(M,L_p)$ in a unique way. Moreover, \cite[Proposition 6.5]{Molino} states that $(\nu_p(M,L_p),\fol_p,g_p^\perp)$ is an infinitesimal foliation, where $g_p^\perp = g(p)|_{\nu_p(M,L_p)}$. We refer to $(\nu_p(M,L_p),\fol_p,g_p^\perp)$ as the \emph{infinitesimal foliation of $\fol$ at $p$}. 

The following local description of the foliation is going to be of use later on:

\begin{prop}[Proposition 2.17 in \cite{Radeschi2017}]\th\label{P: Plaque charts}
Let $(M,\fol,g)$ be a singular Riemannian foliation on an $m$-dimensional manifold, and for  any $p \in M$ we write  $k_p = \dim(L_p)$. Then  for $p\in M$ fixed there exists a coordinate system $(W, \phi)$, with $W\subset M$ and $\phi\colon W\to \R^m$ such that:
\begin{enumerate}[(1)]
\item $p \in W$, $\phi(W) = U \times V$, where $U \subset \R^{k_p}$ and $V \subset \R^{m-k_p}$ are open and bounded subsets containing the origin with smooth boundary.
\item $\phi(p) = (\bar{0},\bar{0})\in U\times V$.
\item $\phi^{-1}(U\times\{\bar{0}\}) = W\cap L_p$.
\item For any $q \in W$, $U \times \mathrm{proj}_2(\phi(q))\subset  \phi(L_{q} \cap W )$.
\item For fixed $(u_0,v_0)\in U\times V$ the curve $\phi(u_0,tv_0)$ is a geodesic of $M$ perpendicular to the leaves of $\fol$.
\end{enumerate}
\end{prop}

\subsection{Holonomy}

In this section we present for $(M,\fol,g)$ a singular Riemannian foliation the notion of the \emph{holonomy} of a closed leaf $L_p\in \fol$  which we simply denote by $L$. We denote by $\nu(L)\to L$ the normal bundle to $L$, and  we have the following construction:
\begin{theorem}[Corollary 1.5 in \cite{MendesRadeschi2019}]\th\label{T:  holonomy map of curve}
Let $L$ be a closed leaf of $(M,\fol,g)$ a singular Riemannian foliation, and let $\gamma\colon [0,1]\to L$ a piecewise smooth curve with $\gamma(0) = p$. Then there exists a map $G\colon [0,1]\times \nu_p(M,L)\to \nu(L)$ such that:
\begin{enumerate}[(i)]
\item $G(t,v) \in \nu_{\gamma(t)}(M,L)$ for every $(t,v)\in [0,1]\times \nu_p(M,L)$.
\item For every $t\in [0,1]$, the restriction $G_t\colon \{t\}\times \nu_p (M,L)\to \nu_{\gamma(t)}(M,L)$ is a linear map; moreover with respect to the infinitesimal foliations the map 
\[
G_t\colon (\nu_p(M,L),\fol_p,g_p^\perp)\to (\nu_{\gamma(t)}(M,L),\fol_{\gamma(t)},g_{\gamma(t)}^\perp)
\]
is a foliated isometry. 
\end{enumerate}
\end{theorem}

Let $\Or(\nu_p(M,L_p),\fol_p)$ denote the group of \emph{foliated isometries} of the infinitesimal foliation. That is, a map $f\in \Or(\nu_p(M,L_p),\fol_p)$ is an isometry $f\colon (\nu_p(M,L_p),g_p^\perp)\to (\nu_p(M,L_p),g_p^\perp)$ such that for any leaf $\mathcal{L}\in \fol_p$ we have $f(\mathcal{L})\in \fol_p$. We denote by $\Or(\fol_p)\subset\Or(\nu_p^\perp(M,L_p),\fol_p)$ the subgroup of foliated isometries that leave the infinitesimal foliation invariant, i.e. $f\in \Or(\fol_p)$ if and only if $f(\mathcal{L})\subset \mathcal{L}$ for all $\mathcal{L}\in \fol_p$. Observe that for a closed piecewise loop $\gamma\colon [0,1]\to L_p$ with start point $p$ we have by \th\ref{T:  holonomy map of curve} a map $G_{\gamma}\colon \nu_p(M,L_p)\to \nu_p(M,L_p)$ with $G_{\gamma}\in \Or(\nu_p(M,L_p),\fol_p)$. Moreover, from \cite[Proposition 2.5]{Corro2023} it follows that the coset $G_{\gamma}\Or(\fol_p)\in \Or(\nu_p(M,L_p),\fol_p)/\Or(\fol_p)$ depends only on the homotopy class of $\gamma$. Thus, we have a well defined map $\rho\colon \pi_1(L_p,p)\to \Or(\nu_p(M,L_p),\fol_p)/\Or(\fol_p)$.

We define the \emph{holonomy of the leaf $L_p$ at $p$} as the image 
\[
\Hol(L_p) = \rho(\pi_1(L_p,p))\subset \Or(\nu_p(M,L_p),\fol_p)/\Or(\fol_p).
\]
Observe that if $q\in L_p$ is another point in $L_p$ and we fix $\alpha\colon I\to L_p$ a path from $q$ to $p$, then we have that $G_{\alpha}\Or(\nu_q(M,L_p),\fol_q)G_{\alpha}^{-1} = \Or(\nu_p(M,L_p),\fol_p)$ and $G_{\alpha}\Or(\fol_q)G_{\alpha}^{-1} = \Or(\fol_p)$. Thus, the holonomy is defined up to conjugation in the group of isometries $\Or(\nu_p(M,L_p))$ of the normal space to the leaf.

A leaf $L_p$ is called \emph{principal} if and only if it is regular and has trivial holonomy. Given a closed singular Riemannian foliation $(M,\fol,g)$, the collection of all principal leaves is called the \emph{principal stratum of $\fol$} and denoted by $M_{\prin}$. We point out that $M_{\prin}\subset M$ is an open and dense subset by \cite[Proposition 2.8]{Corro2023}.

\subsection{\texorpdfstring{$A$}{A}- and \texorpdfstring{$B$}{B}-foliations}

We now consider closed singular Riemannian foliations whose leaves are aspherical manifolds, i.e. such that for $L\in\fol$ we have $\pi_k(L) =0$, for $k\neq 1$. These foliations are called \emph{$A$-foliations}, and where introduced in \cite{GalazGarciaRadeschi2015} as a generalization of torus actions by isometries. The authors showed that on a compact simply-connected manifold $M$, the regular leaves of an $A$-foliation are homeomorphic to tori.

\begin{theorem}[Theorem B in \cite{GalazGarciaRadeschi2015}]\th\label{T: regular leaves are tori}
Let $(M,\fol,g)$ be an $A$-foliation on a compact simply-connected manifold, and let $L\in \fol$ be a regular leaf. Then $L$ is homeomorphic to a torus. 
\end{theorem}

\begin{rmk}
We recall that for $n\geq 5$ by \cite{HsiangShaenson1969}, there  exists $n$-manifolds which are homeomorphic to the $n$-torus, but not diffeomorphic. These manifolds are known as \emph{exotic tori}. We recall that there are examples of $A$-foliations $(M,\fol,g)$ on compact manifolds whose leaves are exotic tori, but $M$ is not simply connected (see \cite{FarrellWu2019}).
\end{rmk}

In \cite{GalazGarciaRadeschi2015} the authors  also introduced the concept of $B$-foliations, which are $A$-foliations whose leaves are homeomorphic to Bieberbach manifolds. Recall that an $n$-dimensional \emph{Bieberbach manifold} $L$ is an aspherical manifold such that $\pi_1(L)$ is isomorphic to  a discrete subgroup $G$ of  $\R^n\rtimes  \mathrm{O}(n)\subset \mathrm{Aff}(n)$ which is torsion free, and such that $\R^n/G$ is compact. Such a group $G$ is called a \emph{Bieberbach group}. We recall that Bieberbach manifolds are those homeomorphic to  smooth manifolds that admit a flat Riemannian metric, i.e. a Riemannian metric whose sectional curvature is equal to $0$ for any $2$-plane in the tangent space at any point. By the following result, we conclude that the leaves of $A$-foliations on compact simply-connected manifolds have fundamental group isomorphic to a Bieberbach group, and except for leaves of dimension $4$, they are homeomorphic to Bieberbach manifolds. 

\begin{theorem}[Corollary 3.2 and Proposition 3.3 in \cite{Corro2023}]
Let $(M,\fol,g)$ be an $A$-foliation on a compact simply-connected manifold. Then the leaves with trivial holonomy are homeomorphic to tori. Moreover, for any leaf $L\in \fol$ with non-trivial holonomy the fundamental group $\pi_1(L,p)$, at $p\in L$, is a Bieberbach group,  and for $\dim(L)\neq 4$ the leaf $L$ is homeomorphic to a Bieberbach manifold. 
\end{theorem}

This implies that, except for the leaves of dimension $4$, from a topological viewpoint $A$-foliations and the $B$-foliations as defined in \cite{GalazGarciaRadeschi2015} are indistinguishable on compact simply-connected manifolds.  Thus in \cite[Remark 3.4]{Corro2023} the author proposed to define \emph{$B$-foliations} as those $A$-foliations $(M,\fol,g)$ such that for any $L\in \fol$ we have that $(L,g|_{L})$ is flat. This definition immediately introduces some rigidity to the diffeomorphism type of the foliation: Combining \th\ref{T: regular leaves are tori} with  the work of Bieberbach \cite{Bieberbach} (and more generally by  \cite[Theorem 3]{FarrellJonesOntaneda2007}) we obtain that for simply-connected manifolds the regular leaves of an $A$-foliation $(M,\fol,g)$ such that $(L,g|_{L})$ is flat, i.e. $B$-foliation as defined in this paragraph, are diffeomorphic to the standard torus.

\begin{rmk}
We point out that a singular Riemannian foliation $(M,\fol,g)$ induced by an effective smooth torus action $\mu\colon T^k\times M\to M$ by isometries is a $B$-foliation, in the sense that the orbits with the induced Riemannian metric are flat. We sketch briefly the proof of this fact.  Fix $p\in M$ such that $T^{k}(p)$ is a principal orbit. Then for the isotropy group at $p$ we have $T^k_p=\{e\}$, where $e\in T^k$ is the identity element (see \cite[Example 3.4.3]{ZillerNotes}). This implies that the orbit $T^k(p)$ is equivariantly diffeomorphic to $T^k$ via the map $\mu_p\colon T^k\to T^k(p)$ given by $\mu_p(\xi)= \mu(\xi,p)$. Via this diffeomorphism the $T^k$-invariant induced metric $g|_{T^k(p)}$ induces a left-invariant metric $\bar{g}_p$ on $T^k$. Since $T^k$ is abelian, by \cite[Corollary 1.3]{Milnor1976} we have that the sectional curvature of $\bar{g}_p$ is non-negative. If it is positive at some point (and thus positive anywhere), then by the Bonnet-Myers Theorem (\cite[Chapter 9, 3.1 Theorem]{doCarmo}) we have that $T^k$ has finite fundamental group, which is a contradiction. Thus, $\bar{g}_p$ is flat, and consequently $(T^k(p),g|_{T^k(p)})$ is a flat manifold. Consider now $q\in M$ not in a principal orbit. We point out that a connected component $(T^k_q)^0$ of the identity element of the isotropy subgroup at $q$ is homeomorphic to a subtorus of $T^k$. Then there exists a finite cover $\overline{T^k(q)}$ of $T^k(q)$ which is homeomorphic to a torus of dimension equal to $k-\dim((T^k_q)^0)$ (see \citep[Example 2.4]{GalazGarciaRadeschi2015}). The metric $g|_{T^k(q)}$ lifts to a $T^k/(T^k_q)^0$-invariant metric $\tilde{g}_q$ on $\overline{T^k(q)}$. Observing that the metric $\tilde{g}_q$ induces a left-invariant metric on the torus $T^k/(T^k_q)^0$, then by applying the same reasoning as for the principal orbit case we conclude that the metric $\tilde{g}_q$ is flat. Since $\overline{T^k(q)}$ is a finite cover we conclude that $g|_{T^k(q)}$ is flat.
\end{rmk}

\subsection{Compact flat metric spaces}\label{S: Flat}

For the sake of completeness, we give a short review of the geometry of flat manifolds. We consider $(M,g)$ to be a compact smooth Riemannian $m$-dimensional manifold, such that at every point $p\in M$ and any $2$-dimensional plane in $T_pM$ we have the sectional curvature equal to $0$, i.e. a flat manifold. By the Cartan-Hadamard theorem (see \cite[Theorem 3.1]{doCarmo}) it follows that the universal cover $\widetilde{M}$ of $M$ is diffeomorphic to $\R^m$. We consider $\widetilde{g}= \pi^\ast(g)$ the pull back metric of $g$ under the covering map $\pi\colon \widetilde{M}\to M$. Observe that by construction $\pi\colon (\widetilde{M},\widetilde{g})\to (M,g)$ is a local isometry, and thus $(\widetilde{M},\widetilde{g})$ is a flat manifold diffeomorphic to $\R^m$. By \cite[Theorem 4.1]{doCarmo} it follows that $(\widetilde{M},\widetilde{g})$ is isometric to $(\R^m,\sigma)$ where $\sigma$ denotes the Euclidean Riemannian metric. Thus we have that there exists a  covering map $\bar{\pi}\colon (\R^m,\sigma)\to (M,g)$ which is a local isometry, i.e. $\bar{\pi}^\ast(g) = \sigma$. By letting $B$ be the group of deck transformations of $\bar{\pi}$, it follows that $B$ is a subgroup of the isometry group $\Iso(\R^m,\sigma)$ which is torsion free, i.e. $B$ is a Bieberbach group. Thus $(M,g)$ is isometric to the homogeneous Riemannian space $(\R^m/B,\bar{\sigma})$,  where $\bar{\sigma}$ is the Riemannian metric induced by the Riemannian submersion $p\colon \R^m\to \R^m/B$ (see \cite[Theorem 1.2.1]{GromollWalschap}).

\subsection{Alexandrov spaces}\label{SS: Alexandrov geometry}
We now recall the definition of an Alexandrov space. Let $(X,d_X)$ be a locally compact and locally complete inner metric space. A point $x\in X$ is said to have \emph{curvature at least $\kappa$}, for some $\kappa\in \R$, if there exist $U_x\subset X$ an open neighborhood of $x$ such that the following holds:
Given any geodesic triangle $\Delta\subset U_x$ there exists a geodesic triangle $\widetilde{\triangle}$ in  the $2$-dimensional surface of constant sectional curvature $\kappa$ denoted by $S^2_{\kappa}$, with edges having the same lengths as the edges of $\triangle$, and such that given $y\in \triangle$ a vertex and $w$ any point in the opposing edge to $y$ in $\triangle$, for  $\tilde{y}$ and $\tilde{w}$ the unique corresponding points in $\widetilde{\triangle}$ to $y$ and $w$ respectively, then we have
\[
d(y,w) \geq d_{S^2_{\kappa}}(\tilde{y},\tilde{w}).
\]
A \emph{local Alexandrov space} $(X,d_X)$ is a locally compact and locally complete inner metric space such that for each point $x\in X$ there exist $\kappa_x\in\R$ so that $x$ has curvature at least $\kappa_x$.

In the case when there exists $\kappa\in \R$ such that any point of $X$ is a point with curvature at least $\kappa$, we say that $(X,d_X)$ is an \emph{Alexandrov space of curvature at least $\kappa$}, and we denote it by $\curv(X)\geq \kappa$. We denote the collection of all Alexandrov spaces with curvature at least $\kappa\in \R$ by $\Alex(\kappa)$.

With respect to the Gromov-Hausdorff topology (see Section~\ref{S: equivariant convergence}) we have the following stability result for Alexandrov spaces:

\begin{theorem}[Perelman's stability Theorem in \cite{Perelman91}, \cite{Kapvitch2007}]\th\label{T: Perelmans stability}
Let $(X,d_X)$ be a compact $n$-dimensional Alexandrov space of $\curv(X) \geq \kappa$. Then there exists an $\epsilon = \epsilon(X) > 0$ such that for any $n$-dimensional Alexandrov space $(Y,d_Y)$ of $\curv(Y)\geq \kappa$ with $d_{GH}(X, Y) < \epsilon$, $Y$ is homeomorphic to $X$.
\end{theorem}

An extremal subset $E$ of an Alexandrov space $(X,d_X)$ is a closed subset subset which is preserved under the gradient flow of $d_X(p, \cdot)$ for all $p \in X$ (see \cite{HarveySearle2021}).

\begin{ex}\th\label{Ex: Orbit type stratification} Let $H$ act effectively smoothly by isometries on a compact Riemannian manifold $(M,g)$, consider $\pi\colon M\to M/H$ and equip $M/H$ with the induced metric $d^\ast$. Then $(M/H,d^\ast)$ is an Alexandrov space. Let $\Gamma\subset H$ be a closed subgroup, and $F_\Gamma$ be the set of points fixed by $\Gamma$. Then by \cite[4.2 Proposition]{PerelmanPetrunin1994} $\pi(F_\Gamma)$ is an extremal subset of $(M/H,d^\ast)$.
\end{ex}

\begin{rmk}\th\label{R: Perelman stability preserves stratification}
The homeomorphism in \th\ref{T: Perelmans stability} can be taken so that it preserves the stratification of $X$ and $Y$ by extremal subsets (see \cite{Kapvitch2007},\cite[Theorem 4.3]{HarveySearle2017}).
\end{rmk}

\subsection{Equivariant convergence}\label{S: equivariant convergence}

We denote by $\mathcal{M}$ the set of all isometry classes of pointed metric spaces $(X,d_X,x)$ such that for each $r>0$ the open ball $B_r(x)$ is relatively compact. We define the following notion of convergence on $\mathcal{M}$: 

We begin by considering first $(X,d_X)$ and $(Y,d_Y)$ metric spaces. We define the \emph{Gromov-Hausdorff distance} between $(X,d_X)$ and $(Y,d_Y)$, denoted by $d_{\mathrm{GH}}((X,d_X),(Y,d_Y))$, to be the infimum  of all Hausdorff distances $d_H(f(X),g(Y))$ for all metric spaces $(Z,d_Z)$  and all isometric embeddings $f\colon X\to Z$ and $g\colon Y\to Z$ (see  \cite[Section 7.3]{BuragoBuragoIvanov}).

We say that a a pointed sequence $\{(X_i,d_{X_i},x_i)\}_{i\in\N}\subset \mathcal{M}$ converges to $(Y,d_Y,y)\in \mathcal{M}$ if for each $r>0$ and each $\varepsilon>0$ there exists $N\in \N$, such that for $n\geq N$ there exists a map $f_n\colon B_r(x_n)\to X$ for which the following hold:
\begin{enumerate}[(i)]
\item $f(x_n) = y$,
\item the $\varepsilon$-neighborhood of $f_n(B_r(x_n))\subset Y$ contains the ball $B_{r-\varepsilon}(y)$,
\item $\mathrm{dis} f_n = \sup\{|d_Y(f_n(z_1),f_n(z_2))-d_{X_n}(z_1,z_2)|\mid z_1,z_2\in B_{r}(x_n)\}<\varepsilon$.
\end{enumerate}
We call this notion of convergence on $\mathcal{M}$, convergence in  the \emph{pointed Gromov-Hausdorff sense}.

\begin{rmk}
When $(X_i,d_{X_i})$ and $(Y,d_Y)$ are  locally compact metric spaces, we have that $(X_i,d_{X_i},x_i)$ converges to $(Y,d_Y,y)$ in the pointed Gromov-Hausdorff sense if and only if for any $R>0$ the open balls $(B_R(x_i),d_{X_i})$ converge to $(B_R(y),d_Y)$ in the Gromov-Hausdorff sense.
\end{rmk}

\begin{rmk}\th\label{R: GH convergence compact spaces implies pGH}
When we consider a sequence of compact spaces $(X_i,d_{X_i})$ and $(Y,d_Y)$, then convergence with respect to the Gromov-Hausdorff distance is equivalent to convergence in the pointed Gromov-Hausdorff sense, in the following sense: in the case when $(X_i,d_{X_i})$ converges in the Gromov-Hausdorff sense to $(Y,d_Y)$, given $y\in Y$ there exists a sequence $\{x_i\mid x_i\in X_i\}_{i\in \N}$ such that $(X_i,d_{X_i},x_i)$ converges in the pointed-Gromov-Hausdorff sense to $(Y,d_Y,y)$ (see \cite[Exercise 8.12]{BuragoBuragoIvanov}).
\end{rmk}

Moreover, for compact spaces we have a nice characterization of convergence in the Gromov-Hausdorff sense. Recall that given $\varepsilon >0$, a subset $S\subset X$ of a metric space is \emph{an $\varepsilon$-net} if we have that $d_X(x,S) = \inf\{d_X(x,s)\mid s\in S\}\leq \varepsilon$. Given $X$, $Y$ two compact metric spaces, and $\varepsilon,\delta>0$, we say that \emph{$X$ and $Y$ are $(\varepsilon,\delta)$-approximations of each other} if there exists $\{x_i\}_{i=1}^N\subset X$, $\{y_i\}_{i=1}^N\subset Y$ such that 
\begin{enumerate}
\item The sets $\{x_i\}_{i=1}^N$, $\{y_i\}_{i=1}^N$ are $\varepsilon$-nets,
\item $|d_X(x_i,x_j)-d_Y(y_i,y_j)|<\delta$ for all $i,j\in \{1,\ldots,N\}$.
\end{enumerate}
When $\varepsilon= \delta$, we say that \emph{$X$ and $Y$ are $\varepsilon$-approximations of each other}. With this we can write the characterization of convergence in the Gromov-Hausdorff sense for compact spaces.

\begin{prop}[Proposition 7.4.11 in \cite{BuragoBuragoIvanov}]\th\label{P: characterization of GH convergence via approximations}
Let $X$ and $Y$ be compact metric spaces. Then we have
\begin{enumerate}[(1)]
\item If $Y$ is an $(\varepsilon,\delta)$ approximation of $X$, then $d_{\mathrm{GH}}(X,Y)< 2\varepsilon+\delta$.
\item If $d_{\mathrm{GH}}(X,Y)<\varepsilon$, then $Y$ is a $5\varepsilon$-approximation of $X$.
\end{enumerate}
\end{prop}

We also consider tuples $(X,d_X,H,x)$,  where $(X,d_X,x)\in \mathcal{M}$ and $H\subset \mathrm{Isom}(X,d_X)$ is a closed subgroup acting effectively on $X$. We say that $(X,d_X,H,x)$ is equivalent to $(Y,d_Y,H',y)$ if there exists $f\colon (X,d_X,x)\to (Y,d_Y,y)$ an isometry with $f(x) = y$, and $\phi\colon H\to H'$ a group isomorphism such that for any $h\in H$ and $x'\in X$ we have $f(h\cdot x') = \phi(h)\cdot f(x')$. We denote by $\mathcal{M}_{\mathrm{eq}}$ the collection of all equivalent tuples $(X,d_X,H,x)$. For $(X,d_X,H,x)\in \mathcal{M}_{\mathrm{eq}}$ and $r>0$ we set $H(r) = \{h\in H\mid h\cdot x\in B_r(x)\}$.

We define the following topology on $\mathcal{M}_{\mathrm{eq}}$: Given $(X,d_X,H,x)$, $(Y,d_Y,H',y)$ in $\mathcal{M}_{\mathrm{eq}}$ and $\varepsilon>0$, we define an \emph{$\varepsilon$-equivariant pointed Gromov-Hausdorff approximation} between $X$ and $Y$ to be  a triple of functions $(f,\phi,\psi)$,
\begin{linenomath}
\begin{align*}
f\colon B_{1/\varepsilon}(x)&\to Y,\\
\phi\colon H(1/\varepsilon)&\to H'(1/\varepsilon),\\
\psi\colon H'(1/\varepsilon)&\to H(1/\varepsilon),\\
\end{align*}
\end{linenomath}
such that 
\begin{enumerate}[(1)]
\setlength\itemsep{0.2em}
\item we have $f(x) = y$.
\item The $\varepsilon$-neighborhood of $f(B_{1/\varepsilon}(x))$ contains $B_{1/\varepsilon}(y)$.
\item For $p,q\in B_{1/\varepsilon}(x)$ we have $|d_Y(f(p),f(p'))-d_X(p,p')|<\varepsilon$.
\item Given $h\in H(1/\varepsilon)$, $p\in B_{1/\varepsilon}(x)$ with $h\cdot p\in B_{1/\varepsilon}(x)$ we have
\[
d_Y(f(h\cdot p),\phi(h)\cdot f(p)) <\varepsilon.
\]
\item For $h'\in H'(1/\varepsilon)$, $p\in B_{1/\varepsilon}(x)$ with $\psi(h')\cdot p\in B_{1/\varepsilon}(x)$, then
\[
d_Y(f(\psi(h')\cdot p),h'\cdot f(p))<\varepsilon.
\]
\end{enumerate}
For $(X,d_X,H,x),(Y,d_Y,H',y)\in\mathcal{M}_{\mathrm{eq}}$ we define the \emph{equivariant pointed Gromov-Haus\-dorff distance}, $d_{\mathrm{eqGH}}\Big((X,d_X,H,x),(Y,d_Y,H',y)\Big)$, to be the infimum of all numbers $\varepsilon>0$ such that there exists an $\varepsilon$-equivariant pointed Gromov-Hausdorff approximation from $(X,d_X,H,x)$ to $(Y,d_Y,H',y)$ and from $(Y,d_Y,H',y)$ to $(X,d_X,H,x)$. This induces on $\mathcal{M}_{\mathrm{eq}}$ the \emph{equivariant pointed Gromov-Hausdorff topology}.

Given a sequence $\{(X_i,d_{X_i},H_i,x_i)\}_{i\in \N}$ and $(Y,d_Y,H',y)\in \mathcal{M}_{\mathrm{eq}}$, we have the following relation between convergence with respect to the equivariant pointed Gromov-Hausdorff distance, and convergence in the pointed Gromov-Hausdorff sense of the pointed quotient spaces, $(X/H,d_X^\ast,x^\ast)$ and $(Y/H',d_Y^\ast,y^\ast)$.

\begin{theorem}[Theorem 2.1 in \cite{Fukaya1986}]\th\label{T: eqpGH convergence implies pGH convergence}
Let  $\{(X_i,d_{X_i},H_i,x_i)\}_{i\in\N}\subset \mathcal{M}_{\mathrm{eq}}$ have limit \linebreak$(Y,d_Y,H',y)\in \mathcal{M}_{\mathrm{eq}}$ with respect to the equivariant pointed Gromov-Hausdorff distance. Then we  have that $(Y/H',d_Y^\ast,y^\ast)$ is the limit of the sequence $\{(X_i/H_i,d_{X_i}^\ast,x_i^\ast)\}_{i\in \N}$ in the sense of  pointed Gromov-Hausdorff convergence.
\end{theorem}

Moreover, we have then the following information on the limit of sequences in $\mathcal{M}_{\mathrm{eq}}$ with respect to the pointed Gromov-Hausdorff convergence:

\begin{theorem}[Proposition~3.6 in \cite{FukayaYamaguchi1992}]\th\label{T: pGH-convergence implies eqGH-convergence}
Assume that for a sequence $\{(X_i,d_{X_i},H_i,x_i)\}_{i\in \N}$ in $\mathcal{M}_{\mathrm{eq}}$ we have that $(Y,d_Y,y)\in \mathcal{M}$ is the limit of $(X_i,d_{X_i},x_i)$ in the pointed Gromov-Hausdorff sense. Then there exist $H'\subset \mathrm{Isom}(Y,d_Y)$ a closed subgroup acting effectively on $Y$, and a subsequence $\{(X_{i_k},d_{X_{i_{
k}}},H_{i_k},x_{i_k})\}_{k\in \N}$ converging to $(Y,d_Y,H',y)$ with respect to the pointed equivariant Gromov-Hausdorff topology. 
\end{theorem}

Moreover in the case of Alexandrov spaces of constant dimension we have the following result:

\begin{theorem}[Proposition~4.1 in \cite{Harvey2016}]\th\label{T: equiv. convergence Alexandrov spaces}
Let 
\[
\Omega_\kappa^n = \{(X,d_X,H,p)\in \mathcal{M}_{\mathrm{eq}}\mid (X,d_X)\in \Alex (\kappa),\: \dime(X) = n,\: H\mbox{ is compact}\}\subset \Alex(\kappa).
\] 
Let $\{(X_i,d_{X_i},H_i,x_i)\}_{i\in \N}\subset\Omega^n_\kappa$ be a sequence, converging in the equivariant pointed Gromov-Hausdorff topology to $(Y,d_Y,H',y)$. Then for a sufficiently large index $i\in \N$, the group $H'$ contains an isomorphic image of $H_i$.
\end{theorem}

\subsection{\texorpdfstring{$\mathsf{N}$}{N}- and \texorpdfstring{$\mathsf{F}$}{F}-structures}\label{SS: N- and F-structures}

The concept of an $\F$-structure (already defined in Section~\ref{S: Introduction}) was introduced by Cheeger and Gromov in \cite{CheegerGromov1986} generalizing the notion of torus actions on manifolds. Moreover they proved that for given $\F$-structure  whose orbits have positive dimension on a compact manifold we can find a family of Riemannian metrics $\{g_t\mid 1\geq t> 0\}$ on the manifold such that the \emph{manifold collapses with bounded curvature}. That is, the injectivity radius of  $(M,g_t)$  goes to $0$ as $t\to 0$, and there exist $\lambda\leq \Lambda$ such that $\lambda \leq \Sec(g_t)\leq \Lambda$ for all $0<t\leq 1$.

In \cite{Fukaya1987,Fukaya1989} Fukaya studied the phenomena of collapse when the sequence $\{(M,g_t)\}$ converges in the Gromov-Hausdorff sense to a Riemannian manifold, and then in \cite{CheegerFukayaGromov1992} it was proven that collapse is controlled by so called \emph{Nilpotent structures}, or \emph{$\Ns$-structures}. We now recall the definition of such a structure.

Given a Riemannian manifold $(M,g)$  of dimension $m$ we consider $\pi\colon F(M)\to M$ to be the $\Or(m)$-principal frame bundle of $M$.  A pure $\Ns$-structure on $(M,g)$ is given by a fibration $\tilde{\eta}\colon F(M)\to B$ such that:
\begin{enumerate}[(i)]
\item The fiber is a nilmanifold isometric to $(N/\Gamma,\nabla^{\mathrm{can}})$, where $N$ is a simply-connected nilpotent group and $\nabla^{\mathrm{can}}$ is the canonical connection on $N$ making all left-invariant vector fields parallel.
\item The structure group of $\tilde{\eta}\colon F(M)\to B$ is contained in the group of affine automorphisms of the fiber.
\item The action of $\Or(m)$ preserves the fibers and the structure group of the fibration $\tilde{\eta}$.
\end{enumerate}  
A pure $\Ns$-structure induces a partition $\mathscr{O}_{\tilde{\eta}}$ of $M$ into submanifolds called \emph{orbits}  which are the images of the $\tilde{\eta}$-fibers under the projection $\pi\colon F(M)\to M$ (observe that the ``orbits'' of $\mathscr{O}$ are not orbits of a group action, but they are defined by orbits of the action of $N$ on $F(M)$). We refer to this decomposition of $M$ as a \emph{pure $\Ns$-structure on $M$}. An $\Ns$-structure has \emph{positive rank} if the orbits have positive dimension.

We observe that the centers of the fibers of a pure $\Ns$-structure on $F(M)$ determine a torus bundle on the nilpotent group $N$. This gives rise to a second $\Or(m)$-invariant fiber bundle $\tilde{f}\colon F(M)\to B_{\tilde{f}}$ whose fiber is a torus, and has affine structure group. We call $\tilde{f}$ the \emph{canonical torus bundle of $\tilde{\eta}$}. We have the following relationship between $\tilde{\eta}$ and $\tilde{f}$ when the fundamental group of $M$ is finite.

\begin{theorem}[Lemma 4.1 and Lemma 3.2 in \cite{Rong1996}]\th\label{T: N-structure over manifold with finite fundamental group are given by T-structure}
An $\Ns$-structure $\tilde{\eta}\colon F(M)\to B_{\tilde{\eta}}$ on $M$ a manifold with finite fundamental group coincides with its canonical torus bundle $\tilde{f}\colon F(M)\to B_{\tilde{f}}$. That is, $B_{\tilde{\eta}} = B_{\tilde{f}}$ and $\tilde{\eta}=\tilde{f}$. Moreover, when $M$ is simply connected then an $\Ns$-structure on $M$ is given by a torus action on $M$.
\end{theorem}

By \cite{CheegerFukayaGromov1992}, given an $\Ns$-structure $\tilde{\eta}\colon F(M)\to B_{\tilde{\eta}}$ there exists a sheaf of Lie algebras on $F(M)$ whose local sections restrict to local right invariant vector fields on the fibers of $\tilde{\eta}$. We say that a Riemannian metric $g$ on $M$ is \emph{$\Ns$-invariant with respect to $\tilde{\eta}$} if these local sections  project under $\pi$ to local Killing fields for the metric $g$. In particular, given $V\subset (M,g)$ a tubular neighborhood of  an orbit of the $\Ns$-structure there exists a normal covering $\widetilde{V}\to V$	admitting an isometric $N$-action.
 
In \cite{CheegerFukayaGromov1992} the authors showed that a sufficiently collapsed manifold $(M,g)$ admits an $\Ns$-structure. And moreover there exists an $\Ns$-invariant Riemannian metric on $M$ related to $g$. 

\begin{theorem}[Theorems 1.3 and 1.7 in \cite{CheegerFukayaGromov1992}]\th\label{T: Existence of N structures on}
For $m>2$ and $D>0$ we denote by $\mathfrak{M}(m,D)$ the class of $m$-dimensional compact connected Riemannian manifolds $(M,g)$ with $|\Sec(g)|\leq 1$, and $\diam(M,g)\leq D$. 

Given $\varepsilon>0$, there exists a constant $v=v(m,D,\varepsilon)>0$ such that given $(M,g)\in \mathfrak{M}(m,D)$ with $\mathrm{Vol}(g)< v$, then $M$ admits a pure $\Ns$-structure $\widetilde{\eta}_\varepsilon\colon F(M) \to B$ of positive rank. Moreover, 
\begin{enumerate}[(a)]
\item there is a smooth metric $g^\varepsilon$ on $M$, which is $\Ns$-invariant for the $\Ns$-structure $\widetilde{\eta}_\varepsilon$,
\item the fibers of $\widetilde{\eta}_\varepsilon$ have diameter less than $\varepsilon$ with respect to $g^\varepsilon$, 
\item and
\[
e^{-\varepsilon} g< g^\varepsilon < e^{\varepsilon}g.
\]
\end{enumerate} 
\end{theorem}

Moreover, the metric $g^\varepsilon$ can be chosen so that the sectional curvatures are close to the ones of $g$ by the following theorem.

\begin{theorem}[Theorem 2.1 in \cite{Rong1996}]\th\label{T: Rong bounds sectional curvature}
Let the assumptions be as in \th\ref{T: Existence of N structures on}. Then the nearby metric $g^\varepsilon$ can  be chosen to satisfy in addition
\[
	\min \Sec(g) -\varepsilon \leq \Sec(g^\varepsilon) \leq\max \Sec(g)+\varepsilon.
\]
\end{theorem}

In particular the metric $g^\varepsilon$ in \th\ref{T: Existence of N structures on}  induces an inner metric $d_{g^\varepsilon}^\ast$ on the  quotient spaces $M/\mathscr{O}_{\tilde{\eta}_\varepsilon}$ of the $\Ns$-structures  making them into Alexandrov spaces of curvature at least $\min \Sec(g)-\varepsilon$.

\begin{ex}\th\label	{Ex: orbit type stratification torus actions}
Given an effective smooth $T^k$-action by isometries on a Riemannian manifold $(M,g)$, let $\Gamma\subset T^k$ be an isotropy group of some point $p\in M$. \emph{The orbit type stratum $\Sigma_{\Gamma}^\ast\subset M/T^k$} correspond to all the orbits in $M/T^k$ that have isotropy $\Gamma$. Observe that the closure of $\Sigma_{\Gamma}^\ast$ corresponds to the set of all orbits whose isotropy group contains $\Gamma$ (See \th\ref{Ex: Orbit type stratification}).
\end{ex}

Combining \th\ref{Ex: orbit type stratification torus actions}, Theorems \ref{T: Perelmans stability}, \ref{T: N-structure over manifold with finite fundamental group are given by T-structure}, \ref{T: Existence of N structures on}, and \ref{T: Rong bounds sectional curvature} we obtain the following result: 

\begin{theorem}[Theorem 1.3 in \cite{PetruninRongTuschmann1999}]\th\label{T: Bounded collapse in 1-connnected manifolds}
Assume that $(M_n, g_n)$ is a sequence of simply-connected compact Riemannian $m$-manifolds with sectional curvature bounds \linebreak$\lambda \leq \Sec(g_n) \leq\Lambda$ and diameters $\diam(g_n) \leq D$ which collapses with bounded curvature. Then, given any $\varepsilon>0$, for $n$ sufficiently large the following holds:
\begin{enumerate}[(a)]
\item There exists a smooth global effective $T^k$ action on $M_n$ with empty fixed-point set, all of whose orbits have diameter less than $\varepsilon$ with respect to $g_n$.
\item There exists on $M_n$ a $T^k$ invariant metric $g^\varepsilon_n$ which satisfies
\[
e^{-\varepsilon}g_n < g^\varepsilon_n < e^\varepsilon g_n,\quad \lambda- \varepsilon \leq \Sec(g^\varepsilon_n) \leq \Lambda +\varepsilon.
\]
\item Moreover assume that $(M_n,g_n)$ converges to an Alexandrov space $X$ of dimension $(m-k)$ in the Gromov-Hausdorff topology. Then for  the orbit space $M_n/T^k$ equipped with the metric induced by $g^\varepsilon_n$, we have that the Gromov-Hausdorff distance between $X$ and $M_n/T^k$ is less than $\varepsilon$.
\item The orbit space $M_n/T^k$ is homeomorphic to $X$, and moreover, this homeomorphism preserves the stratification of $M_n/T^k$ by orbit type.
\end{enumerate}
\end{theorem}

\section{Proofs of Main Theorems}\label{S: Proofs of main theorems}

In this section we give the proofs of some of our main results. We first present the proof of \th\ref{MT: Collapse of flat foliations}. We then present the proof of \th\ref{MT: Approximating flat foliations by N-structures}. With this we are able to give the proof of \th\ref{MT: flat foliations are given by group actions}.

\subsection{Proof of Theorem~\ref{MT: Collapse of flat foliations}}\label{SS: Proof of collapse}

We start by giving a proof of \th\ref{MT: Collapse of flat foliations}. The strategy of the proof is to explicitly define the sequence of collapsing metrics, and then verify that they collapse with bounded curvature and diameter.

\begin{proof}[Proof of \th\ref{MT: Collapse of flat foliations}]
Consider $(M,\fol,g)$ a regular $B$-foliation with leaves of positive dimension on a compact Riemannian manifold. 

We denote by $m$ the dimension of $M$, and we observe that since the leaves have positive dimension for any $q\in M$, we can write $g(q) = g(q)^\top+g(q)^\perp$ where $g(q)^\top$ is the metric $g(q)$ restricted to $T_q L_q\subset T_q M$, and $g(q)^\perp$ is the metric $g(q)$ restricted to $\nu_q(M,L_q)\subset T_q M$. For $\delta>0$ we consider the Riemannian metric $g_\delta$ defined at $q$ as:
\begin{equation}\label{EQ: collapsing sequence}
g_\delta(q) = \delta^2 g(q)^\top+g(q)^\perp.
\end{equation}
We claim that the family $g_\delta$ is the desired collapsing sequence.

We fix $p\in M$ and set $k = \dim(\fol)$. We consider the chart $(W,\phi)$ given by \th\ref{P: Plaque charts} around $p$. For $\overline{p}\in L_p\cap W$ observe that 
\[
\phi(\overline{p}) = (u_1(\overline{p}),\ldots,u_{k}(\overline{p}),v_1(\overline{p}),\ldots,v_{m-k}(\overline{p}))= (u_1(\overline{p}),\ldots,u_{k}(\overline{p}),0,\ldots,0).
\] 
Observe that by construction we have that $\phi^{-1}(U\times\{(0,\ldots,0)\}) = L_p\cap W$.

By assuming that $W$ is small enough so that the normal bundle $\nu(M,L_p)\to L_p$ is trivial over $W\cap L_p$,  and thus we can consider a basis $E_1,\ldots,E_{m-k}\in \nu(M,L_p)|_{W\cap L_p}$ of smooth vector fields on  $L_p\cap W = \phi(U\times\{\bar{0}\})$ normal to $L_p$, such that for $\overline{p}\in L_p\cap W$ we have $g(\overline{p})(E_i,E_j)= \delta_{i,j}$. We take $r>0$ small enough so that $\expo_\nu\colon \nu^{r}(M,L_p)\to M$ is a diffeomorphism onto the open tubular neighborhood $\mathrm{Tub}^{r}(L_p)$ and set $\overline{W} = W\cap \mathrm{Tub}^{r}(L_p)$. For $\overline{p}\in L_p\cap \overline{W}$ we set $u(\overline{p}) = (u_1(\overline{p}),\ldots,u_{k}(\overline{p}),0,\ldots,0)$ as a local chart, and observe that for $q\in \overline{W}$ there exists a unique $\overline{p}\in \overline{W}\cap L_p$ and unique $t_{1},\ldots,t_{m-k}\in\R$ such that 
\[
q = \expo_{\overline{p}}\left(\sum_{j=1}^{m-k} t_j E_j(\overline{p}) \right).
\] 
With these observations we define the \emph{Fermi-coordinates} $\psi=(x,y)\colon \overline{W}\to \overline{U}\times \overline{V}\subset \R^{k}\times \R^{m-k}$ with $x(q)=(\psi_1(q),\ldots,\psi_k(q))$ and $y(q)=(\psi_{k+1}(q),\ldots,\psi_m(q))$ given as follows:
\begin{linenomath}
\begin{align*}
x_a\left(\expo_{\overline{p}}\left(\sum_{j=1}^{m-k} t_j E_j(\overline{p})\right)\right)=\psi_a\left(\expo_{\overline{p}}\left(\sum_{j=1}^{m-k} t_j E_j(\overline{p})\right)\right) &= u_a(\overline{p})\qquad \mbox{for }a=1,\ldots,k,\\
y_i\left(\expo_{\overline{p}}\left(\sum_{j=1}^{m-k} t_j E_j(\overline{p})\right)\right)=\psi_i\left(\expo_{\overline{p}}\left(\sum_{j=1}^{m-k} t_j E_j(\overline{p})\right)\right) &=t_i\qquad \mbox{for }i=k+1,\ldots,m.
\end{align*}
\end{linenomath}

\noindent By \cite[Lemma~2.4]{Gray} we have that the restriction of $\frac{\partial}{\partial y_{k+1}}=\frac{\partial}{\partial \psi_{k+1}},\ldots, \frac{\partial}{\partial y_{m}}=\frac{\partial}{\partial \psi_m}$ to  $L_p\cap \overline{W}$ are orthonormal. Also we observe that $\frac{\partial}{\partial x_a}(q)=\frac{\partial}{\partial \psi_a}(q) = \frac{\partial}{\partial u_a}(q)$, for $a= 1,\ldots,k$ by construction.

For $(x,y) = q\in \overline{W}$ we have the unique $g$-orthogonal decomposition $\frac{\partial}{\partial y_i} = X_i+V_i$ where $X_i\in T_{q} L_{q}$ and $V_i\in \nu_{q}(M,L_{q})$. We define matrices $A(x,y)$, $B(x,y)$, $C(x,y)$, and $D(x,y)$ as follows:
\begin{linenomath}
\begin{align*}
A_{a,b}(x,y)&=g(x,y)\left(\frac{\partial}{\partial x_a},\frac{\partial}{\partial x_b}\right),  & &1\leq a,b\leq k,\\
B_{a,i}(x,y) &=g(x,y)\left(\frac{\partial}{\partial x_a},\frac{\partial}{\partial y_i}\right) = g(x,y)\left(\frac{\partial}{\partial x_a},X_i\right), & &1\leq a\leq k,\: k+1\leq i\leq m,\\
C_{i,j}(x,y)&= g(x,y)(X_i,X_j), & &k+1\leq i,j\leq m,\\
D_{i,j}(x,y)&= g(x,y)(V_i,V_j), & &k+1\leq i,j\leq m.
\end{align*}
\end{linenomath}
Then we have the following local description of $g$ in the $(x,y)$-coordinates:
\[
g(x,y) 
= \begin{pmatrix}
  A(x,y) & B(x,y) \\
  B(x,y) & C(x,y)+D(x,y)
\end{pmatrix}.
\]

Observe that $A(x,y)$, $B(x,y)$, $C(x,y)$  correspond to the components  of the metric $g|_{L_{(x,y)}}$, which is a  homogeneous flat metric on $L_{(x,y)}$ (see Section~\ref{S: Flat}). From this it follows that $g|_{L_{(x,y)}}$ does not depend on base point on $L_{(x,y)}$, and in particular it is independent of the $x$-coordinate. To see this, observe that for $q= (x,y)$ we can lift metric $g|_{L_q}$ to the Euclidean Riemannian metric on the universal cover $\R^{k_q}$ of $L_q$, where $k_q = \dim(L_q)$. Thus locally the lifted metric does not depend on the base point, and thus the same holds for $g|_{L_q}$.

The matrix $D(x,y)$ corresponds to the geometry  normal to $L_q$ at $q=(x,y)$ in $M$. But since $\fol$ is a closed regular Riemannian foliation, this geometry is independent of the point in $L_q$. In particular for $(x',y) = q'\in L_q$ close enough to $q$ in $L_q$ we have $D(x,y) = D(x',y')$. Thus the metric $g$ in the local coordinates $(x,y)$ is independent of $x$. In this way for $(x,y) = q\in \overline{W}$ we can write
\[
g(x,y) = \begin{pmatrix}
  A(y) & B(y) \\
  B(y) & C(y)+D(y)
\end{pmatrix}.
\]

In analogous fashion we have for the metric $g_\delta$ and the indices $a,b=1,\ldots,k$, $i,j=k+1,\ldots,m$ the following identities:
\begin{linenomath}
\begin{align*}
g_\delta(x,y)\left(  \frac{\partial}{\partial u_a},\frac{\partial}{\partial u_b}\right) &= \delta^2 g(x,y)\left(\frac{\partial}{\partial u_a},\frac{\partial}{\partial u_b}\right) =\delta^2 A_{a,b}(x,y),\\
g_\delta(x,y)\left( \frac{\partial}{\partial \bar{u}_a},\frac{\partial}{\partial x_i}\right) &=g_\delta(x,y)\left(\frac{\partial}{\partial u_a},X_i\right) =\delta^2 g(x,y)\left(\frac{\partial}{\partial u_a},X_i\right) = \delta^2 B_{a,i}(x,y),\\
g_\delta(x,y)(X_i,X_j) &= \delta^2g(x,y)(X_i,X_j) = \delta^2 C_{i,j}(x,y),\\
g_\delta(x,y)(V_i,V_j)&= g(x,y)(V_i,V_j) = D_{i,j}(x,y).
\end{align*}
\end{linenomath}
Thus we have locally around $p$ in the $(x,y)$-coordinates that 
\[
g_\delta(x,y)= \begin{pmatrix}
  \delta^2 A(y) & \delta^2 B(y) \\
  \delta^2 B(y) & \delta^2 C(y)+ D(y)
\end{pmatrix}.
\]
We now follow \cite[Proof of Theorem 2.1]{CheegerGromov1986}, and consider the change of coordinates $\bar{x}_a = \delta x_a$. Thus we have $\frac{\partial}{\partial \bar{x}_a} = \frac{1}{\delta}\frac{\partial}{\partial x_a}$, and we get the following local description of $g_\delta$ in the $(\bar{x},y)$-coordinates:
\[
g_\delta(\bar{x},y) = \begin{pmatrix}
  A(y) & \delta B(y) \\
  \delta B(y) & \delta^2 C(y)+ D(y)
\end{pmatrix}.
\]
Since this matrix does not depend on $\bar{x}$, we can define $g_\delta$ on $\R^{k}\times \overline{V}$. As $\delta\to 0$, the metric $g_\delta$ converges smoothly on $\R^{k}\times \overline{V}$ to the metric
\[
\begin{pmatrix}
A(y) & 0 \\
0 & D(y)
\end{pmatrix}.
\]
Moreover, $A(y)$ induces a flat metric on $\R^{k}$. Observe that for $r>0$ small enough so that $B_r(\bar{0})\subset \overline{U}$, and for $\delta$ small enough we have that $g_\delta(x,y) = g_\delta(y)$ restricted to $B_{r}(\bar{0})\times\overline{V}$ is isometric to $g(\bar{x},y) = g(y)$ restricted to $B_{\delta r}(\bar{0})\times\overline{V}$. Now by hypothesis
\[
\lambda\leq \Sec\Big(g(x,y)|_{B_{r}(\bar{0})\times\overline{V}}\Big)\leq \Lambda,
\]
and thus the same holds for $g(\bar{x},y)|_{B_{\delta r}(\bar{0})\times\overline{V}}$. By the compactness of $M$ we get that 
\[
\lambda\leq \Sec(g_\delta)\leq \Lambda.
\]
Moreover, observe that $\diam(M,g_\delta)\leq \diam(M,g)$, since we are ``shrinking some tangent directions'' of $M$.

Given a leaf $L_p$, we observe that since we are not changing the norm of the tangent directions in $TM$ perpendicular to $L_p$, we have that for $r>0$ the tubular neighborhood $B_{r}(L_p,g)$ of $L_p$ with respect to $g$ is the tubular neighborhood $B_{r}(L_p,g_\delta)$ of $L_p$ with respect to $g_\delta$. Moreover we have that  $B_{r}(p,g_\delta)\subset B_{r}(L_p,g_\delta)= B_{r}(L_p,g)$, where $B_{r}(p,g_\delta)$ is the ball of radius $\delta>0$ centered at $p$ with respect to $g_\delta$. But by construction we have that $\mathrm{Vol}(B_{r}(L_p,g), g_\delta)$ goes to $0$ as $\delta\to 0$. Thus we conclude that 
\[
\lim_{\delta\to 0}\mathrm{Vol}(B_{r}(p,g_\delta)) =0.
\]
We assume now that there exists $r_0>0$ such that $\Inj(g_\delta)\geq r_0$. By \cite[Proposition 14]{Croke1980}, taking $r_0/2>r>0$ we have that there exists a constant $C(m)>0$ depending only on $m$ the dimension of $M$ such that for all $\delta$
\[
\mathrm{Vol}(B_{r}(p,g_\delta))\geq C(n) r^m>0.
\] 
But this is a contradiction to the fact that $\mathrm{Vol}(B_{r}(p,g_\delta))$ decreases to $0$. Thus  we have that 
\[
\lim_{\delta\to 0} \Inj(g_\delta)=0.
\]

Thus the metrics $\{g_\delta\}_{\delta>0}$ give a collapsing sequence of $M$ with bounded curvature and diameter.
\end{proof}

\subsection{Proof of Theorem~\ref{MT: Approximating flat foliations by N-structures}}\label{SS: Proof of approximation by N-structures}

In this subsection we present the proof of \th\ref{MT: Approximating flat foliations by N-structures}. The idea of the proof consists in modifying the collapsing sequence from \th\ref{MT: Collapse of flat foliations} to obtain a new sequence of metrics converging to the original foliated metric $g$. Then we prove that the $g$-perpendicular part of action fields to $\fol$ vanishes.

\subsubsection*{Setup for the proof of Theorem~\ref{MT: Approximating flat foliations by N-structures}}\hfill\\
We give the necessary setup for the proof. We take  $(M,\fol,g)$ to be a regular $B$-foliation on a compact $m$-dimensional manifold with finite fundamental group and  $m>2$.  

We consider the collapsing sequence $\{g_\delta\}_{\delta>0}$ given in \eqref{EQ: collapsing sequence} for the proof of \th\ref{MT: Collapse of flat foliations}. Next we fix $p\in M$ and we consider the $(\bar{x},y)$-coordinates  on a small neighborhood $\overline{W}$ around $p$ as given  in the proof of \th\ref{MT: Collapse of flat foliations}. We recall that in these coordinates we have that 
\[
g_\delta(\bar{x},y) = \begin{pmatrix}
A(y) & \delta B(y)\\
\delta B(y) & \delta^2 C(y)+D(y)
\end{pmatrix}.
\]
Abusing notation, taking $\delta=1/n$ we now consider the sequence of Riemannian metrics $g_n = g_{1/n}$ on $M$, and  fix $\varepsilon>0$. Then, by \th\ref{T: Existence of N structures on} exists an $\Ns$-structure  $\widetilde{\eta}_\varepsilon$ on $M$ with orbits of positive dimension, and there exists $N(\varepsilon)\in \N$  such that for $n\geq N(\varepsilon)$ we have an $\widetilde{\eta}_\varepsilon$-invariant Riemannian metric  $g_n^\varepsilon$ on $M$  for which it holds:
\begin{equation}\label{EQ: Equivalence between g and gvarepsilon}
e^{-\varepsilon} g_n < g^\varepsilon_{n} < e^{\varepsilon} g_n\quad\mbox{and}\quad 	 \lambda-\varepsilon \leq \Sec(g^\varepsilon_n) \leq \Lambda+\varepsilon.
\end{equation}

Given $X(p)\in T_pM$, we consider the projections $X^\perp(p)\in \nu_p(M,L_p)$ and $X^\top(p)\in T_pL_p$. Observe that for a vector field $X\in \X(M)$ the vector fields $X^\perp$ and $X^\top$ are smooth vector fields on $M$, since $\fol$ is a smooth foliation.

We define a Riemannian metric $\bar{g}^\varepsilon_{N(\varepsilon)}$ on $M$ as follows:
\begin{linenomath}
\begin{equation}\label{EQ: modifed N_e metric}
\begin{split}
\bar{g}^\varepsilon_{N(\varepsilon)}(X,Y) =& g^\varepsilon_{N(\varepsilon)}(X^\perp,Y^\perp)+N(\varepsilon) g^\varepsilon_{N(\varepsilon)}(X^\perp,Y^\top)\\
&+N(\varepsilon) g^\varepsilon_{N(\varepsilon)}(X^\top,Y^\perp)+ N(\varepsilon)^2 g^\varepsilon_{N(\varepsilon)}(X^\top,Y^\top)
\end{split}
\end{equation}
\end{linenomath}
for $X,Y\in \X(M)$. 

We prove the first part of \th\ref{MT: Approximating flat foliations by N-structures}.

\begin{lemma}\th\label{L: modified metric converge GH to g}
For the metrics defined in \eqref{EQ: modifed N_e metric} the spaces $(M,\bar{g}^\varepsilon_{N(\varepsilon)})$ converge to $(M,g)$ in the Gromov-Hausdorff sense as $\varepsilon\to 0$.
\end{lemma}

\begin{proof}
For a given $Z\in \X(M)$ we consider $Z_\varepsilon = Z^\perp+N(\varepsilon)Z^\top$. Then we have that by definition $g_{N(\varepsilon)} (Z_\varepsilon,Z_\varepsilon)= g(Z,Z)$ and $g^\varepsilon_{N(\varepsilon)}(Z_\varepsilon,Z_\varepsilon) = \bar{g}^\varepsilon_{N(\varepsilon)}(Z,Z)$. Thus from \eqref{EQ: Equivalence between g and gvarepsilon} we have
\[
e^{-\varepsilon}g(Z,Z) = e^{-\varepsilon}g_{N(\varepsilon)} (Z_\varepsilon,Z_\varepsilon)<g^\varepsilon_{N(\varepsilon)} (Z_\varepsilon,Z_\varepsilon)
 <e^{\varepsilon}g_{N(\varepsilon)} (Z_\varepsilon,Z_\varepsilon) = e^{\varepsilon}g(Z,Z).
\]
Since $g^\varepsilon_{N(\varepsilon)} (Z_\varepsilon,Z_\varepsilon) = \bar{g}^\varepsilon_{N(\varepsilon)} (Z,Z)$, this implies that for $Z\in TM$ fixed we have 
\begin{linenomath}
\begin{align}
e^{-\varepsilon}g(Z,Z)\leq \bar{g}^\varepsilon_{N(\varepsilon)}(Z,Z) \leq e^{\varepsilon}g(Z,Z).\label{EQ: Converges of bar g}
\end{align}
\end{linenomath}

Fix $p,q\in M$ and consider $c\colon [0,1]\to M$ the minimizing geodesic from $p$ to $q$ with respect to $g$. Then we have that
\[
e^{-\varepsilon/2} d_{\bar{g}^\varepsilon_{N(\varepsilon)}}(p,q) \leq e^{-\varepsilon/2}\ell(\bar{g}^\varepsilon_{N(\varepsilon)})(c)\leq \ell(g)(c) = d_g(p,q).
\]
Now consider $c_\varepsilon\colon [0,1]\to M$ the minimizing geodesic from  $p$ to $q$ with respect to $\bar{g}^\varepsilon_{N(\varepsilon)}$. Then we have that
\[
d_{\bar{g}^\varepsilon_{N(\varepsilon)}}(p,q) = \ell(\bar{g}^\varepsilon_{N(\varepsilon)})(c_\varepsilon) \geq e^{-\varepsilon/2}\ell(g)(c_\varepsilon)\geq e^{-\varepsilon/2}d_g(p,q).
\]
That is we have 
\[
e^{-\varepsilon/2}d_g(p,q)\leq d_{\bar{g}^\varepsilon_{N(\varepsilon)}}(p,q)\leq e^{\varepsilon/2}d_g(p,q).
\]
This implies that both $\Id_M\colon (M,d_g)\to (M,d_{\bar{g}^\varepsilon_{N(\varepsilon)}})$ and $\Id^{-1}_M\colon (M,d_{\bar{g}^\varepsilon_{N(\varepsilon)}})\to (M,d_g)$ are $e^{\varepsilon/2}$-Lipschitz. 

Recall that the \emph{dilatation} of a Lipschitz function $f\colon (X,d_X)\to (Y,d_Y)$ is $\mathrm{dil}(f)=\sup\{d_Y(f(x_1),f(x_2))/d_X(x_1,x_2)\mid x_1,x_2\in X\}$. Observe that given two different metrics $d_1$, $d_2$ on the same space $X$ such that $\Id_X\colon (X,d_1)\to (X,d_2)$ and $\Id_X^{-1}\colon (X,d_2)\to(X,d_1)$ are Lipschitz, we have that $\mathrm{dil}(\Id_X)$ is not necessarily equal to $\mathrm{dil}(\Id_X^{-1})$ since this quantities depend on the metrics. Then we have that
\[
\mathrm{dil}(\Id_M)\leq e^{\varepsilon/2},\quad\mbox{and}\quad \mathrm{dil}(\Id_M^{-1})\leq e^{\varepsilon/2},
\]
 Thus we conclude that for the Lipschitz distance we have
\[
d_L((M,d_g),(M,d_{\bar{g}^\varepsilon_{N(\varepsilon)}}))\leq \varepsilon/2.
\]
This implies that $(M,d_{\bar{g}^\varepsilon_{N(\varepsilon)}})$ converges Gromov-Hausdorff to $(M,d_g)$ as $\varepsilon$ goes to $0$ (see \cite[Example 7.4.3]{BuragoBuragoIvanov}).
\end{proof}

We recall from \th\ref{T: N-structure over manifold with finite fundamental group are given by T-structure} that our $\Ns$-structures $\widetilde{\eta}_\varepsilon$ are given by torus fibrations with fiber $T^k$ since $M$ has finite fundamental group.

By assuming that the coordinate chart $\overline{W}$ is small enough, we may assume  that for the $\Or(m)$-principal bundle $\pi\colon F(M)\to M$ we have a trivalization $\psi\colon \pi^{-1}(\overline{W})\to \overline{W}\times \Or(m)$. With this trivialization, given $q\in \overline{W}$ we take $\tilde{q} = \psi^{-1}(q,\mathrm{Id})$ to be a lift of $q$, i.e. $\pi(\tilde{q}) = q$. We consider the fiber $\widetilde{\eta}_\varepsilon^{-1}(\widetilde{\eta}_\varepsilon(\tilde{p}))$ of the $\Ns$-structure that contains $\tilde{q}$, and consider a fixed vector $x\in T_e T^k$. We now consider the action field $\widetilde{X}_\varepsilon(\tilde{q})$ at $\tilde{q}$  of $x$ on the fiber $\widetilde{\eta}^{-1}_\varepsilon(\tilde{\eta}_\varepsilon(\tilde{q}))$. We set now $X_\varepsilon(q) = \pi_\ast(\widetilde{X}_\varepsilon(\tilde{q}))$; this vector field is well defined, since the fibration $\widetilde{\eta}_\varepsilon$ is $\Or(m)$-invariant. 

We now proof the second part of \th\ref{MT: Approximating flat foliations by N-structures}.

\begin{lemma}
The vector field $X_\varepsilon(q)$ converges to a direction in $T_q(L_q)$.
\end{lemma}

\begin{proof}
We denote by $X_\varepsilon^\top$ the component tangent to $\fol$ and by $X_\varepsilon^\perp$  the component perpendicular to $L_p$ with respect to $g$. Thus we have that 
\[
e^{-\varepsilon}(\|X_\varepsilon^\perp\|^2_{g_{N(\varepsilon)}})\leq e^{-\varepsilon}(\|X_\varepsilon\|^2_{g_{N(\varepsilon)}}) < \|X_\varepsilon\|^2_{g^\varepsilon_{N(\varepsilon)}}.
\]
Now we observe that $\|X_\varepsilon^\perp\|^2_g = \|X_\varepsilon^\perp\|^2_{g_{N(\varepsilon)}}$. Since the orbits of $\widetilde{\eta}_\varepsilon$-structure have diameter less than $\varepsilon$ and $x$ is fixed, the norm $\|X_\varepsilon\|^2_{g^\varepsilon_{N(\varepsilon)}}$ goes to $0$ as we make $\varepsilon$ tend to $0$. From this together with the facts that $\|X_\varepsilon^\perp\|^2_g\geq 0$ and $\|X_\varepsilon^\top\|^2_{g_{N(\varepsilon)}}\geq 0$, we conclude that $\|X_\varepsilon^\perp\|^2_{g}\to 0$ as $\varepsilon\to 0$.

In the case when $\|X_\varepsilon^\top\|^2_{g}\to 0$, we conclude that $X_\varepsilon$ converges to the origin $\bar{O}\in T_p L_p$. Thus  in this case the tangent space of the $\widetilde{\eta}_\varepsilon$-orbits converge to a subspace of the tangent space to the orbits.
\end{proof}

We finish this section by proving that we have also uniform convergence of the Riemannian metrics $\bar{g}^\varepsilon_{N(\varepsilon)}$ to $g$.

\begin{lemma}\th\label{L: uniform convergence of bar g to g }
The Riemannian metrics $\bar{g}^\varepsilon_{N(\varepsilon)}$ defined in \eqref{EQ: modifed N_e metric} converge uniformly to $g$. 
\end{lemma}

\begin{proof}
As pointed out in the proof of \th\ref{L: modified metric converge GH to g}, given $Z\in \X(M)$ we have by \eqref{EQ: Converges of bar g} that the norms $\|Z\|^2_{\bar{g}^\varepsilon_{N(\varepsilon)}}$ converge to $\|Z\|^2_g$ as $\varepsilon\to 0$.

For $X,Y\in TM$, we use the polarization identities 
\[
\bar{g}^\varepsilon_{N(\varepsilon)}(X,Y) = \frac{1}{4}\left(\|X+Y\|^2_{\bar{g}^\varepsilon_{N(\varepsilon)}}-\|X-Y\|^2_{\bar{g}^\varepsilon_{N(\varepsilon)}}\right),
\]
to obtain
\[
\frac{1}{4}\left(e^{-\varepsilon}\|X+Y\|^2_{g}-e^\varepsilon\|X-Y\|^2_{g}\right)<\bar{g}^\varepsilon_{N(\varepsilon)}(X,Y)<\frac{1}{4}\left(e^{\varepsilon}\|X+Y\|^2_{g}-e^{-\varepsilon}\|X-Y\|^2_{g}\right).
\]
Since 
\[\lim_{\varepsilon\to 0}\frac{1}{4}\left(e^{-\varepsilon}\|X+Y\|^2_{g}-e^\varepsilon\|X-Y\|^2_{g}\right) = \frac{1}{4}\left(\|X+Y\|^2_{g}-\|X-Y\|^2_{g}\right) = g(X,Y)
\]
and 
\[\lim_{\varepsilon\to 0}\frac{1}{4}\left(e^{\varepsilon}\|X+Y\|^2_{g}-e^{-\varepsilon}\|X-Y\|^2_{g}\right) = \frac{1}{4}\left(\|X+Y\|^2_{g}-\|X-Y\|^2_{g}\right) = g(X,Y),
\]
we conclude that $\bar{g}^\varepsilon_{N(\varepsilon)}(X,Y)\to g(X,Y)$ as $\varepsilon\to 0$  at the same rate, independently of the base point $p\in M$.

Consider now a local coordinate system $(x_1,\ldots,x_m)$. Then we have that the coefficients $(\bar{g}^\varepsilon_{N(\varepsilon)})_{ij}(p)$ converge to $g_{ij}(p)$ with the same rate given by the rate of convergence of $e^\varepsilon$, and $e^{-\varepsilon}$ to $1$ as $\varepsilon\to 0$, and thus it is independent of the choice of a base point $p$.  Thus $(M,\bar{g}^\varepsilon_{N(\varepsilon)})$ converges with respect to the $C^0$-topology to $(M,g)$.
\end{proof}

\subsection{Proof of Theorem~\ref{MT: flat foliations are given by group actions}}\label{SS: Proof that B-foliations on simply-connected are given by torus actions}

We finish the manuscript by presenting the proof of \th\ref{MT: flat foliations are given by group actions}. For manifolds of dimension bigger than $2$, the proof consists of showing first that the orbit spaces of the $T^k$-actions induced by the $\Ns$-structures converge in the Gromov-Hausdorff sense to the quotient space $M/\fol$. Then by equivariant Gromov-Hausdorff convergence theory for the $T^k$-actions there exists a limit action by isometries by a Lie group $G$. Then we prove that the tangent distribution to the $G$-principal orbits agrees with the tangent distribution of $\fol$. Lastly, we prove that the orbits of the $G$-action agree with the orbits of the foliation and that $G$ is a torus. The lower dimensional cases are treated separately using existing classification results.

\subsubsection*{Setup for the Proof of Theorem~\ref{MT: flat foliations are given by group actions} for dimensions $>2$}\hfill\\
We present the necessary set up to prove \th\ref{MT: flat foliations are given by group actions} for the case when the dimension of the manifold is $>2$. 

Let $(M,\fol,g)$ be a compact simply-connected Riemannian manifold of dimension $>2$ with a regular $B$-foliation. 

In this case the hypothesis of \th\ref{MT: Collapse of flat foliations} are satisfied, and thus there exists a collapsing sequence of Riemannian metrics $g_\delta$ with bounded curvature defined in \eqref{EQ: collapsing sequence}. Given $\varepsilon>0$, we consider the $\Ns$-structure $\widetilde{\eta}_\varepsilon$ given by \th\ref{T: Existence of N structures on}. By \th\ref{T: N-structure over manifold with finite fundamental group are given by T-structure} it follows that the $\Ns$-structure $\widetilde{\eta}_\varepsilon$  is given by an effective torus actions $\mu_\varepsilon\colon T^k\times M\to M$, and we  denote by $T^k_\varepsilon\subset \mathrm{Diff}(M)$ the image of $T^k$ under the map $T^k\ni\xi\mapsto \mu_\varepsilon(\xi,\cdot)\in \mathrm{Diff}(M)$. 

Moreover, by \th\ref{T: Existence of N structures on} for each $\varepsilon>0$ we have a Riemannian metric $(M,g^\varepsilon_{N(\varepsilon)})$ which is $\mu_\varepsilon$-invariant. In the proof of \th\ref{MT: Approximating flat foliations by N-structures} we modified these metrics $g^\varepsilon_{N(\varepsilon)}$ to define in \eqref{EQ: modifed N_e metric} a sequence of Riemannian metrics $\bar{g}^\varepsilon_{N(\varepsilon)}$, which by \th\ref{L: uniform convergence of bar g to g } they converge uniformly to the Riemannian metric $g$. We also highlight that by construction, the metrics $\bar{g}^\varepsilon_{N(\varepsilon)}$ are $\mu_\varepsilon$-invariant.

We have the following lemma:

\begin{lemma}\th\label{Cl: Convergence of modified orbit spaces converges to leaf space}
The sequence $(M/T^k_\varepsilon,d^\ast_{\bar{g}^\varepsilon_{N(\varepsilon)}})$ converges to $(M/\fol,d_g^\ast)$ in the Gromov-Haus\-dorff sense as $\varepsilon\to 0$.
\end{lemma}

\begin{proof}
Let $\lambda\in \R$ be  the lower sectional curvature bound of $(M,g)$. Given $\varepsilon>0$, we have by construction that $(M, g^\varepsilon_{N(\varepsilon)})$ has $\Sec(g^\varepsilon_{N(\varepsilon)})\geq  \lambda-\varepsilon$. Thus the space $(M/T^k_\varepsilon,d^\ast_{g^\varepsilon_{N(\varepsilon)}})$ is an Alexandrov space with $\curv\geq \lambda-\varepsilon$. Therefore, given $\varepsilon_0>0$, for  $0<\varepsilon\leq\varepsilon_0$ we have that $(M/\fol,d_g^\ast)$ and $(M/T^k_\varepsilon, d^\ast_{g^\varepsilon_{N(\varepsilon)}})$ are compact Alexandrov spaces with $\curv\geq \lambda-\varepsilon_0$. Moreover, they both have the same dimension, and by construction $(M/T^k_\varepsilon, d^\ast_{g^\varepsilon_{N(\varepsilon)}})$ converges to $(M/\fol,d^\ast_g)$ in the Gromov-Hausdorff sense.  

We fix $p^\ast\in M/\fol$, and consider $q^\ast\in M/\fol$ close enough to $p^\ast$ in $M/\fol$, such that there is a unique geodesic $\gamma^\ast$ in $M/\fol$ joining $p^\ast$ to $q^\ast$. By \cite[Lemma 2.2]{GrovePetersen1991} we can find geodesics $\gamma^\ast_\varepsilon$ in $(M/T^k_\varepsilon,d^\ast_{g^\varepsilon_{N(\varepsilon)}})$ starting at $p^\ast$ which converge uniformly to $\gamma^\ast$. We denote by $q_\varepsilon^\ast$ the other endpoint of the geodesic $\gamma^\ast_\varepsilon$. Observe that by construction the sequence $q_\varepsilon^\ast$ converges to $q^\ast$ as $\varepsilon\to 0$. 

We now consider lifts $q_\varepsilon$ and $\bar{q}_\varepsilon$ in $M$ of $q^\ast_\varepsilon$ for the projection map $\pi_\varepsilon\colon M\to M/T^k_\varepsilon$ such that $d_{g^\varepsilon_{N(\varepsilon)}}(p,q_\varepsilon)= d^\ast_{g^\varepsilon_{N(\varepsilon)}}(p^\ast,q^\ast_\varepsilon)$ and $d_{\bar{g}^\varepsilon_{N(\varepsilon)}}(p,\bar{q}_\varepsilon)= d^\ast_{\bar{g}^\varepsilon_{N(\varepsilon)}}(p^\ast,q^\ast_\varepsilon)$. Thus, there exists $V_\varepsilon\in \nu^{g^\varepsilon_{N(\varepsilon)}}_p(M,T^k_\varepsilon)$ and $\bar{V}_\varepsilon\in \nu^{\bar{g}^\varepsilon_{N(\varepsilon)}}_p(M,T^k_\varepsilon)$ such that $\exp_p^{g^\varepsilon_{N(\varepsilon)}}(V_\varepsilon) =q_\varepsilon$ and $\exp_p^{\bar{g}^\varepsilon_{N(\varepsilon)}}(\bar{V}_\varepsilon) =\bar{q}_\varepsilon$. By construction there exists $\xi_\varepsilon\in T^k_\varepsilon$ such that $\xi_\varepsilon q_\varepsilon = \bar{q}_\varepsilon$. We point out that $\xi_\varepsilon$ can be factored as $\xi_\varepsilon=\bar{\xi}_\varepsilon\cdot\tilde{\xi}_\varepsilon$, where $\bar{\xi}_\varepsilon\in (T^k_\varepsilon)_p$ and $\tilde{\xi}_\varepsilon\not\in (T^k_\varepsilon)_p$. In particular  since $T^k$ is abelian, we have that there exists an action field $X^\ast_\varepsilon\in T_p T^k_\varepsilon(p)$ (determined by $\tilde{\xi}_\varepsilon$) such that  $V_\varepsilon = D_p\bar{\xi}_\varepsilon(\bar{V}_\varepsilon)+X_\varepsilon^\ast$. Since $T^k$ is compact, and the element $X^\ast_\varepsilon$ is determined by the exponential map of $T^k$, we conclude that $X^\ast_\varepsilon$  converges in $T_pM$, up to a subsequence, to a vector $X$. Moreover, by Theorem~\ref{MT: Approximating flat foliations by N-structures}  the tangent spaces of the $T^k_\varepsilon$-orbits converge to the tangent space of the leaves of $\fol$, and thus we have that $X$ is tangent to the leaf $L_p$. 

We also have that  $d_{g^\varepsilon_{N(\varepsilon)}}(p,q_\varepsilon) = d^\ast_{g^\varepsilon_{N(\varepsilon)}}(p^\ast,q^\ast_\varepsilon)$, which converges to $d^\ast_g(p^\ast,q^\ast)$ as $\varepsilon$ goes to $0$. This implies that the sequence $\{q_\varepsilon\}\subset M$ is eventually contained in a tubular neighborhood of $L_p$ with respect to $g$. In particular the sequence $\{q_\varepsilon\}$ is eventually bounded with respect to $g$. This implies that there exists $q_0\in M$ such that, up to a subsequence, $q_\varepsilon$ converges to $q_0$ with respect to $g$.

We consider $V_0\in T_p M$ to be the vector such that $\exp_p^g(V_0) =q_0$. We observe that since $V_\varepsilon = D_p\bar{\xi}_\varepsilon(\bar{V}_\varepsilon)+X^\ast_\varepsilon$, and both $V_\varepsilon$ and $X^\ast_\varepsilon$ converge, up to a subsequence, then $D_p\xi_\varepsilon(\bar{V}_\varepsilon)$ also converges, up to a subsequence, to some element $\bar{V}\in T_p M$. Moreover we have that $V = \bar{V}+X$ by construction. 

Since $\exp_p^{g^\varepsilon_{N(\varepsilon)}}(tV_\varepsilon)$ projects to the geodesic $\gamma_\varepsilon^\ast$, we conclude that $q_\varepsilon^\ast$ converges to $q_0^\ast$ and thus we have $q_0^\ast = q^\ast$. Given that for any $t\in[0,1]$ we have that $\exp_p^{g^\varepsilon_{N(\varepsilon)}}(tV_\varepsilon)$ converges to $\exp_p^g(tV)$ with respect to $g$, the same argument as before shows that $(\exp_p^g(tV))^\ast$ converges to $\gamma^\ast(t)$. This implies that the geodesic  $\exp_p^g(tV)$ is a lift of $\gamma^\ast$.  From this it follows that $d_g(p,q_0) = d_g^\ast(p^\ast,q^\ast)$, and that $\|V\|^2_g = d^\ast_g(p^\ast,q^\ast)$.

Since  $V$ realizes the distance between $L_p$ and $L_{q_0}$, then $V$ is $g$-orthogonal to $L_p$. But we also have that $V=\bar{V}+X$ with $X$ tangent to the leaf $L_p$, possibly equal to $\bar{0}$. This implies that $X=\bar{0}$, and thus $\bar{V}=V$. 

Moreover, by construction we have that $\|\bar{V}_\varepsilon\|^2_{\bar{g}^\varepsilon_{N(\varepsilon)}}=\|D_p\xi_\varepsilon(\bar{V}_\varepsilon)\|^2_{\bar{g}^\varepsilon_{N(\varepsilon)}}$ converges to $\|\bar{V}\|^2_g=\|V\|^2_g$. But since $\|\bar{V}_\varepsilon\|^2_{\bar{g}^\varepsilon_{N(\varepsilon)}}$ $= d^\ast_{\bar{g}^\varepsilon_{N(\varepsilon)}}(p^\ast,q^\ast_\varepsilon)$, we conclude that $d^\ast_{\bar{g}^\varepsilon_{N(\varepsilon)}}(p^\ast,q^\ast_\varepsilon)$ converges to $d_g^\ast(p^\ast,q^\ast)$.

Since $q^\ast$ is arbitrary, we conclude that for sufficiently small balls around $p^\ast$, the same holds. I.e. the function $d^\ast_{\bar{g}^\varepsilon_{N(\varepsilon)}}(p^\ast,-)$ approximates the function $d_g^\ast(p^\ast,-)$. By the triangle inequality, and the fact that the points $q_\varepsilon^\ast$ converge to $q$ with respect to $d^\ast_{\bar{g}^\varepsilon_{N(\varepsilon)}}$ and $d^\ast_g$ we have that the metrics $d^\ast_{\bar{g}^\varepsilon_{N(\varepsilon)}}$ converge uniformly to $d^\ast_g$ over the sufficiently small balls around $p^\ast$ in $M/\fol$. Since $M/\fol$ is compact, we conclude that the metrics $d^\ast_{\bar{g}^\varepsilon_{N(\varepsilon)}}$ converge uniformly to $d^\ast_g$. Thus, we conclude that $(M/T^k_\varepsilon,d^\ast_{\bar{g}^\varepsilon_{N(\varepsilon)}})$ converges to $(M/\fol,d^\ast_g)$ in the Gromov-Hausdorff sense.
\end{proof}

We now use results from Section~\ref{S: equivariant convergence} to proof the following lemma.

\begin{lemma}\th\label{L: existence of limit group action for torus actions}
There exists a Lie group $G$ acting by isometries on $(M,g)$ such that $(M/G,d^\ast_g)$ is isometric to $(M/\fol,d^\ast_g)$.
\end{lemma}

\begin{proof}
Since by \th\ref{L: modified metric converge GH to g}  $(M,d_{\bar{g}^\varepsilon_{N(\varepsilon)}})$ converges in the Gromov-Hausdorff sense to \linebreak$(M,d_g)$, then we have by \th\ref{R: GH convergence compact spaces implies pGH} and \th\ref{T: pGH-convergence implies eqGH-convergence} that there exists a closed  group $G$ acting effectively on $(M,d_g)$ by distance preserving isometries. By the work of Myers and Steenrod \cite{MyersSteenrod1939} it follows that $G$ is a closed subgroup of $\mathrm{Isom}(M,g)$, the Lie group of isometries of $g$. In particular $G$ is a Lie group.

Moreover by \th\ref{T: eqpGH convergence implies pGH convergence} we have that $(M/T^k_\varepsilon,d^\ast_{\bar{g}^\varepsilon_{N(\varepsilon)}})$ converges to $(M/G,d_g^\ast)$ in the Gromov-Hausdorff sense. By \th\ref{Cl: Convergence of modified orbit spaces converges to leaf space}, we have by the uniqueness of the limit that $(M/G,d_g^\ast)$ is isometric to $(M/\fol,d^\ast_g)$.
\end{proof}

For the group action from \th\ref{L: existence of limit group action for torus actions} we obtain the following conclusion.

\begin{lemma}\th\label{Cl: Tangent spaces of G principal orbits agree with tangent spaces of leaves}
Over the set of $G$-principal orbits, we have $T_pG(p)= T_pL_p$. 
\end{lemma}

\begin{proof}
Since $(M/\fol,d^\ast_g)$ is an Alexandrov space with $\curv \geq \lambda$, then $(M/G,d^\ast_g)$ is an Alexandrov space with the same lower curvature bound. Thus given $\varepsilon_0>0$, for $\varepsilon_0\geq \varepsilon>0$ we have that $(M/T^k_{g^\varepsilon_{N(\varepsilon)}})$ and $(M/G,d^\ast_g)$ are Alexandrov spaces with $\curv \geq \lambda-\varepsilon_0$. 

We consider $G(p)$ a principal orbit. Then there exists a sufficiently small ball around $G(p)$ in  $(M/T^k_{g^\varepsilon_{N(\varepsilon)}})$ consisting of only $G$-principal orbits. Let $G(q)$ be an orbit sufficiently close to $G(p)$, and let $\gamma^\ast_G$ be the minimizing geodesic in the principal stratum of $(M/G,d^\ast_g)$ joining $G(p)$ to $G(q)$. Then by \cite[Lemma 2.2]{GrovePetersen1991} there exists minimizing geodesics $\tilde{\gamma}^\ast_\varepsilon$ in the principal stratums of $(M/T^k_{\varepsilon},d^\ast_{g^\varepsilon_{N(\varepsilon)}})$ converging uniformly to $\gamma^\ast_G$.

Take $\delta>0$ small enough so that the $\delta$-ball around $p$ in $(M,g)$ consists of only principal orbits, and consider $B_\delta^\perp(p)\subset (M,g)$ the set of all points $x$ in $M$ with $d_{g}(p,x)<\delta$ such that there exists a unique minimizing geodesic $\gamma\colon [0,1]\to (M,g)$  with $\gamma(0) = p$, $\gamma(1) = x$, and $\gamma'(0)$ perpendicular to $T_pG(p)$ with respect to $g$ (for example we can take $\delta$ smaller than the injectivity radius of $g$ at $p$). Observe that by the Slice Theorem for group actions and how the metric $d^\ast_g$ is defined on $M/G$,  the open ball $B_\delta(G(p))\subset (M/G,d^\ast_g)$ is isometric to $B_\delta^\perp(p)$ under the quotient map $\pi_G\colon (M,d_g)\to (M/G,d_g^\ast)$. Under this identification, the lift $\gamma_G$ in $B_\delta^\perp$ of $\gamma^\ast_G$ is completely determined by the lift $p$ of $G(p)$. 

We also point out that since $(M/T^k_\varepsilon,d^\ast_{g^\varepsilon_{N(\varepsilon)}})$ converges in the Gromov-Hausdorff sense to $(M/G,d^\ast_g)$, by \th\ref{R: Perelman stability preserves stratification} we may assume that $p$ is contained in a $T^k_\varepsilon$-principal orbit. Thus via the quotient map $\pi_{\varepsilon}\colon (M,d_{g^\varepsilon_{N(\varepsilon)}})\to (M/T^k_\varepsilon,d^\ast_{g^\varepsilon_{N(\varepsilon)}})$, the open ball $B^{\varepsilon}_\delta(p^\ast)$ is isometric to $B_\delta^{\varepsilon\, \perp}(p)\subset (M,d_{g^\varepsilon_{N(\varepsilon)}})$, the collection of all points in $B^\varepsilon_\delta(p)\subset M$ such that there exist a minimizing geodesic $\alpha_\varepsilon\colon [0,1]\to (M,g)$ starting at $p$,  and $\alpha_\varepsilon'(0)$ perpendicular to $T_pT^k_\varepsilon(p)$ with respect to $g^\varepsilon_{N(\varepsilon)}$. We also point out that the lift $\tilde{\gamma}\varepsilon$ to $B_{\delta,\varepsilon}^\perp$ of $\tilde{\gamma}^\ast_\varepsilon$ is completely determined by the lift $p$ of $p^\ast$.

Thus we have that for each geodesic $\gamma_G\colon [0,1]\to M$ contained in $B_\delta^\perp(p)\subset (M,d_{g})$ there exist geodesics $\tilde{\gamma}_\varepsilon$ contained in $B_\delta^{\varepsilon\, \perp}(p)\subset (M,d_{g^\varepsilon_{N(\varepsilon)}})$ converging uniformly to $\gamma_G$.  

As stated in the proof of \th\ref{Cl: Convergence of modified orbit spaces converges to leaf space}, the sequence of vectors $\gamma_\varepsilon'(0)$ converges to a vector $V$ that is perpendicular to $L_p$ with respect to $g$. But $V$ is the tangent vector of a lift of the geodesic $\gamma^\ast$ starting at $p$. Thus we have that $\gamma'(0)$ is a scalar multiple of $V$ and thus perpendicular to $L_p$. In other words, we have that $\nu^g_p(M,G(p))\subset \nu^g_p(M,L_p)$.
 
Since $M/G$ is homeomorphic to $M/\fol$, we conclude that 
\[
\dime(\nu^g_p(M,G(p))=\dime(M/G) = \dime(M/\fol) =\dime( \nu^g_p(M,L_p)).
\] 
From this  we conclude that $\nu^g_p(M,G(p))=\nu^g_p(M,L_p)$, and thus $T_pG(p)=T_pL_p$.
\end{proof} 

We now prove the following auxiliary lemma.

\begin{lemma}\th\label{L: singular Riemannian foliation determined on principal part}
Let $M$ be a smooth manifold, and let $\fol'$ and $\fol$ be smooth regular foliations of dimension $k$ on $M$. Assume that $\fol$ and $\fol'$ agree on $D\subset M$ a dense subset of $M$, i.e. for $p\in D$ we have that $L_p = L_p'$ where $L_p\in \fol$ and $L_p'\in \fol'$ are the respective leaves through $p$ of $\fol$ and $\fol'$. Then $\fol = \fol'$, that is the foliations agree everywhere.
\end{lemma}

\begin{proof}
We assume that $m = \dim(M)$ and  for $k=\dime(\fol)=\dime(\fol')$ we consider $\Gr(k,m)\to M$ to be the fiber bundle over $M$ whose fiber at $q\in M$ is the Grassmannian $\Gr(k,T_q M)$ of $k$-dimensional subspaces in $T_q M$. Observe that the total space $\Gr(k,m)$ is a Hausdorff space. 

We recall  that the foliations  $\fol$ and $\fol'$ induce smooth distributions $\Delta$ and $\Delta'$ respectively, by setting $\Delta(q) = T_q L_q$ and $\Delta'(q) = T_q L_q'$. We point out that we can consider these distributions as continuous functions $\Delta,\, \Delta'\colon M\to \Gr(k,m)$. By hypothesis these functions agree on a dense subset, and are continuous. 

We now show by contradiction that $\Delta$ and $\Delta'$ agree everywhere. Assume that they do not agree on some point $q\in M$, i.e. $\Delta(q)\neq \Delta'(q)$. Since $X$ is Hausdorff, there exists open disjoint subsets $U,\, V\subset X$ with $\Delta(q)\in U$ and $\Delta'(q)\in V$. Now we consider $W = \Delta^{-1}(U)\cap (\Delta')^{-1}(V)$ which is open and non-empty, since $q\in W$. We observe that there exists $p\in D\cap W$, by the fact that $D$ is dense. Thus we have that $\Delta(p) = \Delta'(p)$. But this contradicts the fact that $U$ and $V$ are disjoint. Thus $\Delta(q)=\Delta'(q)$ for all $q\in M$. 

This in turn implies that for all $q\in M$ we have $T_q L_q = T_q L_q'$. But observe that by construction we have that $L_q$ is an integral submanifold of $\Delta$, and $L_q'$ is an integral submanifold of $\Delta'$. By the uniqueness of such integral submanifolds it follows that $L_q = L_q'$.
\end{proof}

\begin{lemma}
For the group action of $G$ on $(M,g)$ given by \th\ref{L: existence of limit group action for torus actions}, we have that $L_x=G(x)$ and $G$ is a torus. 
\end{lemma}

\begin{proof}
We begin by pointing out that since $\fol$ is a regular Riemannian foliation, we have that $M/\fol$ is an orbifold. Given that $M/G$ is isometric to $M/\fol$ by \th\ref{L: existence of limit group action for torus actions}, we conclude that $M/G$ is an orbifold. Now by the Slice Theorem (see \cite[Theorem 3.57]{AlexandrinoBettiol}), a sufficiently small ball of $G(x)\in M/G$ in  $(M/G,d^\ast_G)$ is homeomorphic to $\R^m/G_p$. From this and the fact that $M/G$ is an orbifold, we conclude that $G_p$ is a finite subgroup. This implies that the Riemannian foliation on $M$ induced by the partition into the $G$-orbits is a regular Riemannian foliation. Moreover by \th\ref{Cl: Tangent spaces of G principal orbits agree with tangent spaces of leaves} the distributions $T_pL_p$ and $T_pG(p)$ agree over the open dense set of principal $G$-orbits. Then by \th\ref{L: singular Riemannian foliation determined on principal part} we have  $G(x) =L_x$ for an arbitrary $x\in M$.

To finish the proof we consider a fixed $x\in M$. Since we have that $G(x)$ is homeomorphic to $G/G_x$, and $G_x$ is finite, then $G$ is a covering space of $G(x)$. But $G(x)$ is also homeomorphic to $L_x$, which is a torus. This implies that $G$ is a finite cover of a torus and thus $G$ it is aspherical with fundamental group isomorphic to $\pi_1(L_x)$. Thus $G$ is a torus.
\end{proof}

From the discussion above, we have proven the following lemma.

\begin{lemma}\th\label{L: case of dimension > 2}
Let $(M,\fol,g)$ be a regular $B$-foliation on a compact-simply connected Riemannian manifold of dimension $>2$. Then $\fol$ is given by a torus action by isometries on $(M,g)$.
\end{lemma}

\begin{proof}[Proof of \th\ref{MT: flat foliations are given by group actions}]
We consider $(M,\fol,g)$ to be a compact simply-connected $m$-di\-men\-sio\-nal manifold  equipped with a closed non-trivial regular Riemannian foliation such that for $L\in\fol$ we have that $(L,g|_{L})$ is flat; i.e. $\fol$ is a regular $B$-foliation. 

By \th\ref{L: case of dimension > 2} we have proven \th\ref{MT: flat foliations are given by group actions} in the case when $m>2$. We now prove the missing lower dimensional cases.

In the case when $m = 1$ there are no such foliations: Assume that $(M,\fol,g)$ is a regular $B$-foliation on a compact simply-connected $1$-dimensional manifold. Since the leaves have positive dimension, we conclude that the foliation consists of only one leaf, i.e. $\fol = \{M\}$, which is a contradiction to $\fol$ being non-trivial. 

In the case of $m = 2$, we have several cases based on the codimension of the foliation. In the case when the foliation has codimension $0$, this implies that the foliation consists of only one leaf, i.e. $\fol = \{M\}$, which is a contradiction to the non-triviality of $\fol$. The case of codimension $2$ is excluded, because this implies that a leaf has dimension $0$, contradicting the assumption that the leaves of $\fol$ have positive dimension. In the case when $\fol$ has codimension $1$, from  \cite[Theorem D]{GalazGarciaRadeschi2015} it follows that given an $A$-foliation of codimension $1$ on a $2$-dimensional simply-connected manifold $M$, then $M$ is diffeomorphic to $\Sp^2$ and $\fol$ is given by a circle action. This is a foliation with singular leaves and contradicts our assumption of regularity.
\end{proof}

\bibliographystyle{siam}
\bibliography{References}

\end{document}